\documentclass[a4paper,12pt,oneside]{article}
\usepackage{graphicx}

\usepackage[centertags]{amsmath}
\usepackage{graphicx}
\usepackage{amsfonts}
\usepackage{amssymb}
\usepackage[all]{xy}
\usepackage{color}
\usepackage{enumerate}

\usepackage[utf8]{inputenc}
\usepackage{tikz-cd}
\usepackage{hyperref}
\usepackage{mathtools}
\usepackage{import}

\usepackage{relsize}
\usepackage[all]{xy}
\usepackage{microtype}
\usepackage{xspace}
\usepackage{soul}
\usepackage{xcolor}

\usepackage{geometry}

\def\squarebox#1{\hbox to #1{\hfill\vbox to #1{\vfill}}}
\newcommand{\qed}{\hspace*{\fill}
\vbox{\hrule\hbox{\vrule\squarebox{.667em}\vrule}\hrule}\smallskip}

\newtheorem{theorem}{Theorem}[section]
\newtheorem{lemma}[theorem]{Lemma}
\newtheorem{corollary}[theorem]{Corollary}
\newtheorem{proposition}[theorem]{Proposition}
\newtheorem{definition}[theorem]{Definition}
\newtheorem{example}[theorem]{Example}

\newenvironment{proof}{\noindent {\bf Proof:}}{\hfill $\qed $ \newline}

\newcommand{\R}{{\mathbb R}}
\newcommand{\N}{{\mathbb N}}
\newcommand{\Z}{{\mathbb Z}}

\newcommand{\F}{\mathbb{F}}
\newcommand{\B}{{\mathbb B}}

\newcommand{\simto}{\stackrel{\sim}{\longrightarrow}}
\renewcommand{\o}[1]{{\overline{#1}}}

\renewcommand{\S}{{\mathcal S}}
\renewcommand{\B}{{\mathcal B}}

\DeclareMathOperator{\spa}{span}
\DeclareMathOperator{\diag}{diag}

\newcommand{\ad}{{\rm ad}}
\newcommand{\Ad}{{\rm Ad}}

\newcommand{\Sl}{{\rm Sl}}
\newcommand{\SO}{{\rm SO}}

\newcommand{\cl}{\mathrm{cl}}
\newcommand{\e}{{\rm e}}

\newcommand{\g}{\mathfrak{g}}
\renewcommand{\k}{\mathfrak{k}}
\newcommand{\s}{\mathfrak{s}}
\renewcommand{\a}{\mathfrak{a}}

\newcommand{\m}{\mathfrak{m}}

\newcommand{\n}{\mathfrak{n}}

\newcommand{\p}{\mathfrak{p}}

\renewcommand{\sl}{\mathfrak{sl}}

\newcommand{\wt}[1]{{\widetilde{#1}}}

\newcommand{\la}{\lambda}

\begin{document}
\thispagestyle{empty}

\title{Order of dynamical and control systems \\ on maximal compact subgroups}

\author{Mauro Patrão \and Laércio dos Santos}

\maketitle

 \begin{abstract}
In this paper, we characterize the dynamical orders of minimal Morse components and, partially, of control sets on maximal compact subgroups of a semisimple Lie group as an algebraic order similar to the Bruhat order of the Weyl group. We call this order the \emph{extended Bruhat order of the extended Weyl group (or Weyl-Tits group)}, since it projects onto the Bruhat order and is obtained in an analogous way as the Bruhat order was historically obtained, as the incidence order of the Schubert cells in the maximal flag manifold, but now the Schubert cells in the \emph{maximal extended flag manifold}.
 \end{abstract}



\section{Introduction}

The dynamics and semigroup actions induced by elements of a connected semisimple Lie group acting on some of its compact homogeneous spaces are extensively studied in the literature. For example, the iteration of a linear isomorphism acting on the projective space is a classical topic, with applications to the study of skew product flows on projective bundles (Selgrade's Theorem, see \cite{sel}). However, some interesting results, such as the normal hyperbolicity of the minimal Morse components on the projective space, were established only recently, using techniques from Lie groups. More generally, the dynamics and semigroup actions of translations of elements of a linear connected semisimple Lie group acting on a generalized flag manifold are reasonably understood, and there are applications of these results that allow us to obtain the full generalization of Selgrade's Theorem to skew product flows on generalized flag bundles (see \cite{psm1,psm2}). This broader context of generalized flag manifolds encompasses other interesting cases such as the classical flag manifolds of real or complex nested subspaces and also symplectic grassmanians, which were extensively studied in the literature (see for example \cite{ammar, ayala, batterson, kleinsteuber, hermann, pss, shub}). Despite this, there are many interesting dynamics of translations on compact homogeneous spaces of a connected semisimple Lie group which are not generalized flag manifolds. For example, spheres or grassmanians of oriented subspaces in the case of the special linear group. But all these examples of dynamics are projections of dynamics of translations on the respective maximal compact subgroups. 

In \cite{ps1}, the results about dynamics of translations on flag manifolds presented in \cite{dkv,fps,ps} are extended to dynamics of translations on maximal compact subgroups, while, in \cite{patrao-santos}, the results about semigroup actions on flag manifolds presented in \cite{sm93,smt} are extended to semigroup actions on maximal compact subgroups. The description of the minimal Morse components of the action of a given arbitrary element and the description of the control sets of the action of a given arbitrary open semigroup on the maximal flag manifold of a semisimple Lie group are given through the elements of its Weyl group, which are used as labels. In this context, the dynamical orders of minimal Morse components and of control sets are given by the Bruhat order of the Weyl group (see \cite{smo}). On the other hand, in the case of maximal compact subgroups of a semisimple Lie group, it is necessary to replace the previous labels by the elements of a bigger finite group which projects onto the Weyl group.

In this paper, we characterize the dynamical orders of minimal Morse components and, partially, of control sets on maximal compact subgroups of a semisimple Lie group as an algebraic order similar to the Bruhat order of the Weyl group. More precisely, let $K$ be a maximal compact subgroup of a connected semisimple Lie group $G$ with finite center. We look at $K$ as the homogeneous space $G/AN$, where $G = KAN$ is the associated Iwasawa decomposition, and hence the action of $G$ on $K$ is given through the natural action of $G$ on $G/AN$. Instead of the Weyl group $W$, we use the finite group $U$ of connected components of the normalizer of $A$ in $K$ in order to label the minimal Morse components and the control sets. This group was called \emph{extended Weyl group} by Tits in \cite{tits} (in fact \emph{Groupe de Weyl étendu} in French). Since this name is already used in the literature to name another class of groups, we can also call $U$ the \emph{Weyl-Tits group of $G$}. We have that $W$ is isomorphic to $U/C$, where
$C$ is the subgroup of $U$ given by the connected components of the centralizer of $A$ in $K$, and hence there exist a natural homomorphism $\pi$ from $U$ to $W$. The algebraic order of $U$ which describes the dynamical orders of minimal Morse components and control sets projects onto the Bruhat order of $W$ and is obtained in an analogous way as the Bruhat order was historically obtained, as the incidence order of the Schubert cells in the maximal flag manifold (see \cite{ehresmann}). We call this order the \emph{extended Bruhat order of $U$}.

The first main result of the present article is the algebraic characterization in $U$ of the incidence order of the Schubert cells in the maximal compact subgroup. A Schubert cell of $K$ is the closure of the unstable manifold of a minimal Morse component of the action of a regular element. Since the minimal Morse components are parametrized by the elements of $U$, the same is true for the Schubert cells, which are denoted by $\S(u)$, where $u \in U$. For $u,v \in U$ we denote $u \leq v$ if and only if  $\S(u) \subset \S(v)$. In \cite{ps2}, the attaching maps introduced in \cite{lonardo} to calculate the cellular homology of the maximal flag manifolds were lifted to calculate the cellular homology of the maximal compact subgroups in the case of split real forms. And this allowed the algebraic characterization in $U$ of the incidence order $\leq$. Here we extend this approach to an arbitrary semisimple Lie group. For each simple root $\alpha$ of $\g$, the Lie algebra of $G$, there exist a suitable element $s_\alpha$ in $U$ such that $\pi(s_\alpha) = r_\alpha$, the element of $W$ given by the reflection about the hyperplane perpendicular to $\alpha$. Since every element of $W$ can be written as a product of elements $r_\alpha$, the elements of $U$ can be written as a product of elements $s_\alpha$ times some element of $C$.

\begin{theorem}\label{theorem1}
Let $u, u' \in U$. Then $u' \leq u$ if and only if for some (or equivalently for each) reduced expression $\pi(u) = r_1 \cdots r_d$ and $u = s_1 \cdots s_dc$, for some $c \in C$, then $u'=s_1^{k_1}\cdots s_d^{k_d}c$, where
\[
k_i=
\left\{
\begin{array}{lcc}
     0 ~or~ 2, & if & i \in \{i_1, \ldots, i_l\} \\
     1, & if & i \notin \{i_1, \ldots, i_l\}
\end{array}
\right.
\]
with $w_0=\pi(u),$ $w_l=\pi(u')$ and $w_k$
is a reduced expression for each $0\leq k \leq l$, where
\[
w_k=\prod_{i \notin \{i_1,\ldots,i_k\}}r_i.
\]
\end{theorem}

The second main result of the present article is the algebraic characterization in $U$ of the dynamical order of the minimal Morse components of an arbitrary element $g \in G$ acting on the maximal compact subgroup. These minimal Morse components of $g$ coincide with the minimal Morse components of $h = \exp(H)$, where $H \in \cl \mathfrak{a}^+$, the hyperbolic component of the multiplicative Jordan decomposition of $g$, and are in bijection with the elements of the quotient $U_H\backslash U$, where $U_H$ is the subgroup of the elements of $U$ contained in the connected component of the identity of the centralizer of $H$ in $K$. We denote by $\S^H(u)$ the Schubert cell of the minimal Morse component parametrized by $U_Hu$, where $u \in U$, which is the closure of the respective unstable manifold, and denote $U_Hu \leq_H U_Hv$ if and only if $\S^H(u) \subset \S^H(v)$, which is the inverse of the dynamical order of the minimal Morse components.

\begin{theorem}\label{theorem2}
The inverse of the order of the minimal Morse components of $g \in G$ acting on $K$ is characterized by the following equivalent conditions:
    \begin{enumerate}
        \item $u\leq_H v$.
        \item $\S^H(u) \subset \S^H(v)$.
        \item For every $u' \in U_Hu$, there exists $v'\in U_Hv$ such that $\S(u') \subset \S(v')$.
        \item For every $u' \in U_Hu$, there exists $v'\in U_Hv$ such that $u'\leq v'$.
    \end{enumerate}
\end{theorem}

Our third main result is a partial algebraic characterization in $U$ of the dynamical order of control sets of semigroup $S \subset G$ with nonempty interior acting on the maximal compact subgroup. We also establishes a bijection between the control sets on the maximal compact subgroup and the elements of a quotient of $U$ similar to what happens with control sets on the maximal flag manifold (see \cite{smt, smo}). The control sets are parametrized by the elements of $U$ and denoted by $\Bbb{D}(u)$, where $u \in U$ and the dynamical order of control sets is also denoted by $\leq$.

\begin{theorem}\label{theorem3}
The subset
\[
U(S) = \{u \in U: \Bbb{D}(u) = \Bbb{D}(1)\}
\]
is a subgroup of $U$ and the control sets of $S$ are in bijection with the elements of the quotient $U(S)\backslash U$. Furthermore, if $U(S)u\leq U(S)v$, then $\Bbb{D}(v) \leq \Bbb{D}(u)$, where $U(S)u\leq U(S)v$ if and only if for every $v'$ in $U(S)v$ there exists $u'$ in $U(S)u$ such that $u' \leq v'$. Conversely, if $\Bbb{D}(u) \leq \Bbb{D}(v)$ and $s_1\cdots s_d \in U(S)u$ and $\pi(s_1\cdots s_d)=r_1 \cdots r_d$ is a reduced expression, then $s_1^{k_1}\cdots s_d^{k_d} \in U(S)v$, for some  $k_i\in \{0,1,2,3\}$.
\end{theorem}

The structure of the article is the following. In Section 2, we set the notation and recall the necessary definitions and results about semisimple Lie theory, dynamics, and semigroup actions. In Section 3, we prove Theorems \ref{theorem1} and \ref{theorem2}, while, in Section 4, we prove Theorem \ref{theorem3}. We end this introduction with some examples which illustrate our main results.

\begin{example}\label{ex:so3-dinamica}\rm
Let $G = \Sl(3,\R)$ and maximal compact subgroup $K = \SO(3)$, with the other Iwasawa components given by the subgroup $A$ of diagonal matrices with positive diagonal elements and the subgroup $N$ of upper triangular matrices with diagonal elements equal to one. We have that $U$ has 24 elements which are given by permutation matrices with signs and determinant one such as
\[
\begin{pmatrix}
 0 & 1 &  0 \\
 0 & 0 & -1 \\
-1 & 0 &  0
\end{pmatrix}
\]
while
\[
C
=
\left\{
\begin{pmatrix}
1 & 0 & 0 \\
0 & 1 & 0 \\
0 & 0 & 1
\end{pmatrix},
\begin{pmatrix}
-1 & 0 & 0 \\
0 & -1 & 0 \\
0 & 0 & 1
\end{pmatrix},
\begin{pmatrix}
1 & 0 & 0 \\
0 & -1 & 0 \\
0 & 0 & -1
\end{pmatrix},
\begin{pmatrix}
-1 & 0 & 0 \\
0 & 1 & 0 \\
0 & 0 & -1
\end{pmatrix}
\right\}.
\]
Calling
\[
s_1 =
\begin{pmatrix}
1 & 0 & 0 \\
0 & 0 & -1 \\
0 & 1 & 0
\end{pmatrix}
\qquad
\mbox{and}
\qquad
s_2 =
\begin{pmatrix}
0 & -1 & 0 \\
1 & 0  & 0 \\
0 & 0  & 1
\end{pmatrix}
\]
it follows that
\[
C = \{1, ~s_1^2, ~s_2^2, ~s_1^2 s_2^2\}.
\]
The order $\leq$ on $U$ is given in the following diagram (see \cite{ps2}), where $u'\leq u$ if there exists a directed path of arrows starting from $u$ and ending in $u'$.

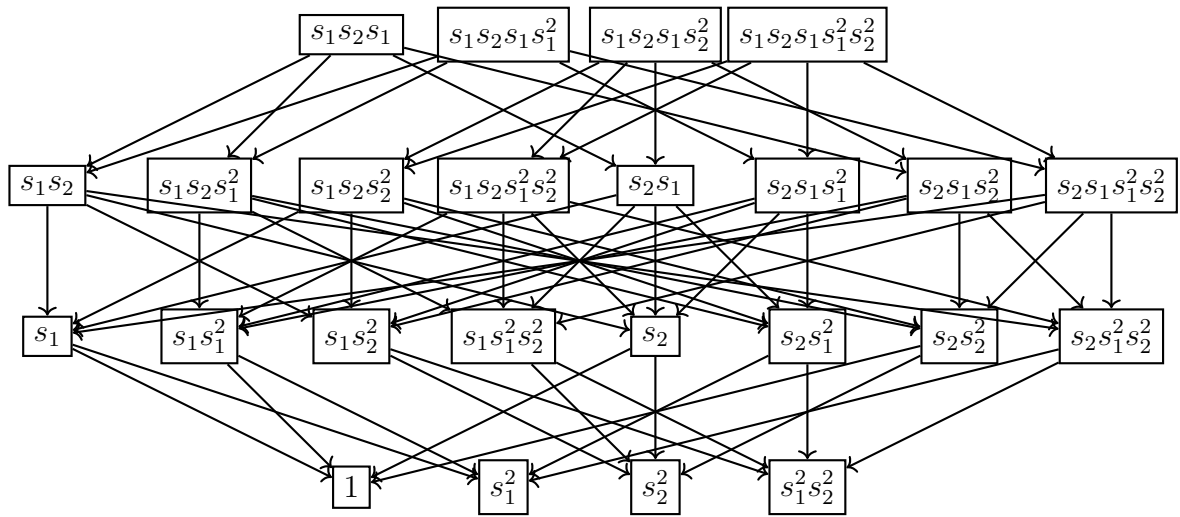
\begin{center}
\begin{tikzpicture}[node distance={2cm}, thick, main/.style = {draw} ]
\node[main] (1) {$s_1s_2s_1$};
\node[main] (2) [right of=1] {$s_1s_2s_1s_1^2$};
\node[main] (3) [right of=2] {$s_1s_2s_1s_2^2$};
\node[main] (4) [right of=3] {$s_1s_2s_1s_1^2s_2^2$};
\node[main] (5) [below of=1] {$s_1s_2s_2^2$};
\node[main] (6) [below of=2] {$s_1s_2s_1^2s_2^2$};
\node[main] (7) [below of=3] {$s_2s_1$};
\node[main] (8) [below of=4] {$s_2s_1s_1^2$};
\node[main] (9) [below of=5] {$s_1s_2^2$};
\node[main] (10) [below of=6] {$s_1s_1^2s_2^2$};
\node[main] (11) [below of=7] {$s_2$};
\node[main] (12) [below of=8] {$s_2s_1^2$};
\node[main] (13) [below of=9] {$1$};
\node[main] (14) [below of=10] {$s_1^2$};
\node[main] (15) [below of=11] {$s_2^2$};
\node[main] (16) [below of=12] {$s_1^2s_2^2$};
\node[main] (17) [left of=5] {$s_1s_2s_1^2$};
\node[main] (18) [left of=17] {$s_1s_2$};
\node[main] (19) [right of=8] {$s_2s_1s_2^2$};
\node[main] (20) [right of=19] {$s_2s_1s_1^2s_2^2$};
\node[main] (21) [left of=9] {$s_1s_1^2$};
\node[main] (22) [left of=21] {$s_1$};
\node[main] (23) [right of=12] {$s_2s_2^2$};
\node[main] (24) [right of=23] {$s_2s_1^2s_2^2$};

\draw[->] (1) -- (18);
\draw[->] (1) -- (17);
\draw[->] (1) -- (7);
\draw[->] (1) -- (19);
\draw[->] (2) -- (18);
\draw[->] (2) -- (17);
\draw[->] (2) -- (8);
\draw[->] (2) -- (20);
\draw[->] (3) -- (5);
\draw[->] (3) -- (6);
\draw[->] (3) -- (7);
\draw[->] (3) -- (19);
\draw[->] (4) -- (5);
\draw[->] (4) -- (6);
\draw[->] (4) -- (8);
\draw[->] (4) -- (20);
\draw[->] (18) -- (22);
\draw[->] (18) -- (9);
\draw[->] (18) -- (11);
\draw[->] (18) -- (24);
\draw[->] (17) -- (21);
\draw[->] (17) -- (10);
\draw[->] (17) -- (12);
\draw[->] (17) -- (23);
\draw[->] (5) -- (22);
\draw[->] (5) -- (9);
\draw[->] (5) -- (12);
\draw[->] (5) -- (23);
\draw[->] (6) -- (21);
\draw[->] (6) -- (10);
\draw[->] (6) -- (11);
\draw[->] (6) -- (24);
\draw[->] (7) -- (22);
\draw[->] (7) -- (10);
\draw[->] (7) -- (11);
\draw[->] (7) -- (12);
\draw[->] (8) -- (21);
\draw[->] (8) -- (9);
\draw[->] (8) -- (11);
\draw[->] (8) -- (12);
\draw[->] (19) -- (21);
\draw[->] (19) -- (9);
\draw[->] (19) -- (23);
\draw[->] (19) -- (24);
\draw[->] (20) -- (22);
\draw[->] (20) -- (10);
\draw[->] (20) -- (23);
\draw[->] (20) -- (24);
\draw[<-] (13) -- (22);
\draw[<-] (13) -- (21);
\draw[<-] (13) -- (11);
\draw[<-] (13) -- (23);
\draw[<-] (14) -- (22);
\draw[<-] (14) -- (21);
\draw[<-] (14) -- (12);
\draw[<-] (14) -- (24);
\draw[<-] (15) -- (9);
\draw[<-] (15) -- (10);
\draw[<-] (15) -- (11);
\draw[<-] (15) -- (23);
\draw[<-] (16) -- (9);
\draw[<-] (16) -- (10);
\draw[<-] (16) -- (12);
\draw[<-] (16) -- (24);
\end{tikzpicture}
\end{center}

Now consider the dynamics given by the translations on $\SO(3)$ of the following element
\[
g^t
=
\begin{pmatrix}
\e^{2t} & 0 & 0 \\
0 & \e^{-t} & t\e^{-t} \\
0 & 0 & \e^{-t}
\end{pmatrix}.
\]
Its Jordan decomposition is given by $g^t = h^tu^t$, where
\[
h^t
=
\exp(tH)
=
\begin{pmatrix}
\e^{2t} & 0 & 0 \\
0 & \e^{-t} & 0 \\
0 & 0 & \e^{-t}
\end{pmatrix},
\qquad
\mbox{with}
\qquad
H
=
\begin{pmatrix}
2 & 0 & 0 \\
0 & -1 & 0 \\
0 & 0 & -1
\end{pmatrix},
\]
and
\[
u^t
=
\begin{pmatrix}
1 & 0 & 0 \\
0 & 1 & t \\
0 & 0 & 1
\end{pmatrix}.
\]
Since
\[
K_H^0
=
\left\{
\begin{pmatrix}
1 & 0 & 0 \\
0 & \cos \alpha & -\sin \alpha \\
0 & \sin \alpha & \cos \alpha
\end{pmatrix}
: \alpha \in \R
\right\},
\]
which is isomorphic to $\SO(2)$, and $U_H=\{1, s_1, s_1^2, s_1^3\}$, there are 6 minimal Morse components given by $K_H^0u$, for $u \in U_H \backslash U$, all diffeomorphic to the circle. The recurrent set is given by ${\rm fix} (h^t) \cap {\rm fix} (u^t)$, which has 12 elements, two in each minimal Morse component. For example, inside $K_H^0$, the recurrent set is given by $\{1, s_1^2\}$. The six elements of $U_H \backslash U$ are given by
\[
U_H = \{1, s_1, s_1^2, s_1^3\},
\quad
U_Hs_2^2 = \{s_2^2, s_1s_2^2, s_1^2s_2^2, s_1^3s_2^2\}
\]
\[
U_Hs_2 = \{s_2, s_1s_2, s_2s_1^2s_2^2, s_1s_2s_1^2s_2^2\},
\quad
U_Hs_2s_1^2 = \{s_2s_1^2, s_1s_2s_1^2, s_2^3, s_1s_2^3\}
\]
\[
U_Hs_2s_1 = \{s_2s_1, s_1s_2s_1, s_2s_1s_2^2, s_1s_2s_1s_2^2\},
\]
\[
U_Hs_2s_1^3 = \{s_2s_1^3, s_1s_2s_1^3, s_2s_1^3s_2^2, s_1s_2s_1^3s_2^2\}
\]
and the order $\leq $ on $U_H \backslash U$ is given by
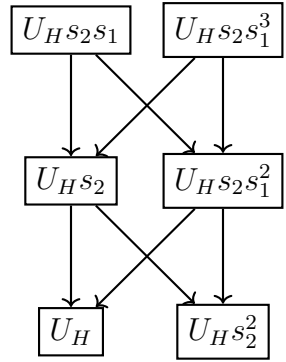
\begin{center}
\begin{tikzpicture}[node distance={2cm}, thick, main/.style = {draw} ]
\node[main] (1) {$U_Hs_2s_1$};
\node[main] (2) [right of=1] {$U_Hs_2s_1^3$};
\node[main] (3) [below of=1] {$U_Hs_2$};
\node[main] (4) [right of=3] {$U_Hs_2s_1^2$};
\node[main] (5) [below of=3] {$U_H$};
\node[main] (6) [right of=5] {$U_Hs_2^2$};

\draw[->] (1) -- (3);
\draw[->] (1) -- (4);
\draw[->] (2) -- (3);
\draw[->] (2) -- (4);
\draw[->] (3) -- (5);
\draw[->] (3) -- (6);
\draw[->] (4) -- (5);
\draw[->] (4) -- (6);
\end{tikzpicture}
\end{center}
where $U_Hu'\leq U_Hu$ if there exists a directed path of arrows starting from $U_Hu$ and ending in $U_Hu'$.
\qed
\end{example}

\begin{example}\rm
    Let $G = \Sl(3,\R)$ and maximal compact subgroup $K = \SO(3)$, with the other notation in Example \ref{ex:so3-dinamica}. Let $S = \mathrm{Sl}^+(3,\R)$ be the semigroup of matrices, in $G$, with positive entries. We have that $S$ is an open semigroup of $G$, has 3 control sets on the maximal flag manifold of $\Sl(3,\R)$ and $W(S)=\{1,r_1\}$, where $r_1=\pi(s_1)$ (see \cite{smt}, Example 5.4 and \cite{gfsm}, Corollary 3). Furthermore, $S$ has 6 control sets on $\SO(3)$ and $C(S)=\{1,s_1^2\}$ (see \cite{patrao-santos}, Example 3.28). Thus
    \[
    U(S)=\{1,s_1,s_1^2,s_1^3\}.
    \]
    Moreover, the order on $U(S) \backslash U$ coincides with the order on $U_H \backslash U$, in Example \ref{ex:so3-dinamica}, and the order $\leq $ of the control sets of $S$ on $\SO(3)$ contains the arrows shown in the following diagram
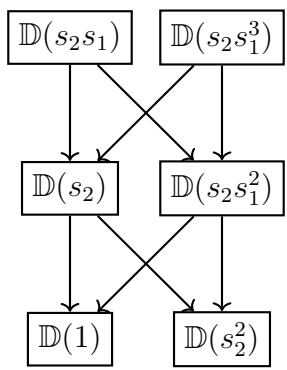
\begin{center}
\begin{tikzpicture}[node distance={2cm}, thick, main/.style = {draw} ]
\node[main] (1) {$\Bbb{D}(s_2s_1)$};
\node[main] (2) [right of=1] {$\Bbb{D}(s_2s_1^3)$};
\node[main] (3) [below of=1] {$\Bbb{D}(s_2)$};
\node[main] (4) [right of=3] {$\Bbb{D}(s_2s_1^2)$};
\node[main] (5) [below of=3] {$\Bbb{D}(1)$};
\node[main] (6) [right of=5] {$\Bbb{D}(s_2^2)$};

\draw[->] (1) -- (3);
\draw[->] (1) -- (4);
\draw[->] (2) -- (3);
\draw[->] (2) -- (4);
\draw[->] (3) -- (5);
\draw[->] (3) -- (6);
\draw[->] (4) -- (5);
\draw[->] (4) -- (6);
\end{tikzpicture}
\end{center}
where $\Bbb{D}'< \Bbb{D}$ if there exists a directed path of arrows starting from $\Bbb{D}'$ and ending in $\Bbb{D}$. It is possible that control sets $\Bbb{D}(s_2)$ and $\Bbb{D}(s_2s_1^2)$ are related, and the same for $\Bbb{D}(s_2s_1)$ and $\Bbb{D}(s_2s_1^3)$.
\qed
\end{example}

\begin{example}\rm
    Let $G=\SO(2,4)_0$ be the connected component of the identity of the group $\SO(2,4)$. A Cartan decomposition of the Lie algebra $\mathfrak{so}(2,4)$ is given by $\mathfrak{so}(2,4)=\mathfrak{k}\oplus\mathfrak{s}$ where $\mathfrak{k}\simeq \mathfrak{so}(2)\oplus \mathfrak{so}(4)$ and
    \[
    \mathfrak{s} 
        =\left\{ 
        \left( 
        \begin{array}{cc} 
        0 & A \\ 
        A^t & 0 \\
        \end{array}
        \right)
        : A ~\text{is a}~ 2\times 4 ~\text{matrix}\right\}.
    \]
    Let $E_{rs}$ be the $6\times 6$ matrix with the only nonzero entry 1 in position $(r,s)$. If $H_1=E_{13}+E_{31}$ and $H_2=E_{24}+E_{42}$, then $\mathfrak{a}=\spa\{H_1,H_2\}$ is a maximal abelian subspace in $\mathfrak{s}$. For each $i \in \{1,2\}$, let $\lambda_i$ be defined by $\lambda_i(a_1H_1+a_2H_2)=a_i$. Thus, $\Sigma=\{\lambda_1-\lambda_2, \lambda_2\}$ is a simple system of roots of $\mathfrak{so}(2,4)$.
    If $r_1$ and $r_2$ denote the orthogonal reflections with respect to $\lambda_1-\lambda_2$ and $\lambda_2$, respectively, then
    \[
    r_1(a_1H_1+a_2H_2)=a_2H_1+a_1H_2 ~~\text{and}~~r_2(a_1H_1+a_2H_2)=a_1H_1-a_2H_2
    \]
    and the Weyl group is given by
    \[
    W=\{1, ~r_1, ~r_2, ~r_1r_2, ~r_2r_1, ~r_1r_2r_1, ~r_2r_1r_2, ~(r_1r_2)^2=(r_2r_1)^2\}. 
    \]
    If $s_1$ and $s_2$ are the block diagonal matrices
    \begin{equation*}
        s_1=\diag\left(k_1,k_1,I_2\right) ~~\text{and}~~ s_2=\diag\left(I_2,k_2,k_2\right)
    \end{equation*}
    where 
    \begin{equation*}
        k_1 =
        \left( 
        \begin{array}{cc} 
        0 & 1 \\ 
        -1 & 0 \\ 
        \end{array} 
        \right),~~
        k_2=
        \left( 
        \begin{array}{cc} 
        1 & 0 \\ 
        0 & -1 \\ 
        \end{array} 
        \right)
        ~\text{ and }~
        I_2=
        \left( 
        \begin{array}{cc} 
        1 & 0 \\ 
        0 & 1 \\ 
        \end{array} 
        \right),
    \end{equation*}
    then $\pi(s_1)=r_1$, $\pi(s_2)=r_2$, $s_1^2=\diag\left(-I_2,-I_2,I_2\right)\in C$ and $s_2^2=I_6$, where $I_6$ is the $6 \times 6$ identity matrix. Thus, $C=\{1,s_1^2\}$ and is contained in the center of $U$, since $C$ is a normal subgroup of $U$. Furthermore, since $(s_1s_2)^2=(s_2s_1)^2$
    it follows that the 16 elements of $U$ are given by
    \[
    1, ~s_1^2, ~s_1, ~s_1^3, ~s_2, ~s_2s_1^2, ~s_1s_2, ~s_1s_2s_1^2, ~s_2s_1, ~s_2s_1^3,
    \]
    \[
    s_1s_2s_1, ~s_1s_2s_1^3, ~s_2s_1s_2, ~s_2s_1s_2s_1^2, ~(s_1s_2)^2, ~(s_1s_2)^2s_1^2
    \]
    and the order $<$ on $U$ is given by
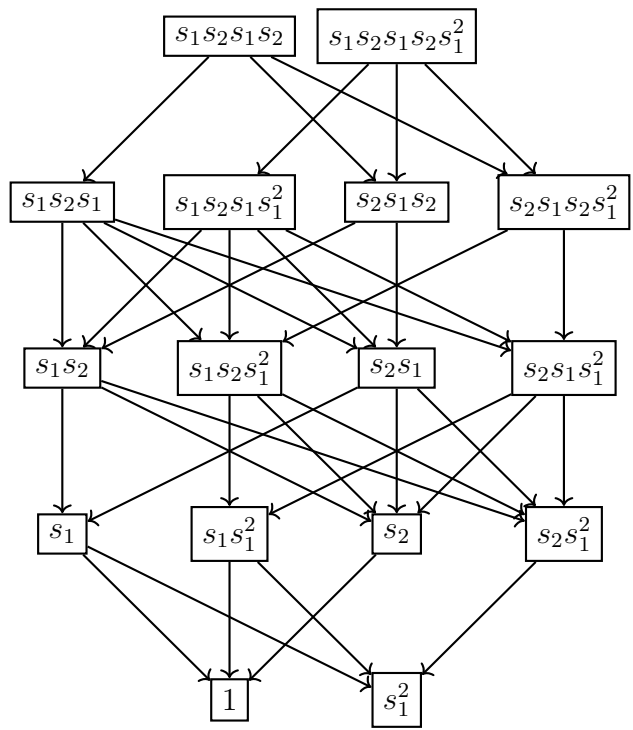
\begin{center}
\begin{tikzpicture}[node distance={2.2cm}, thick, main/.style = {draw} ]
\node[main] (1) {$s_1s_2s_1s_2$};
\node[main] (2) [right of=1] {$s_1s_2s_1s_2s_1^2$};
\node[main] (3) [below of=1] {$s_1s_2s_1s_1^2$};
\node[main] (4) [below of=2] {$s_2s_1s_2$};
\node[main] (5) [below of=3] {$s_1s_2s_1^2$};
\node[main] (6) [below of=4] {$s_2s_1$};
\node[main] (7) [below of=5] {$s_1s_1^2$};
\node[main] (8) [below of=6] {$s_2$};
\node[main] (9) [below of=7] {$1$};
\node[main] (10) [below of=8] {$s_1^2$};
\node[main] (11) [left of=3] {$s_1s_2s_1$};
\node[main] (12) [right of=4] {$s_2s_1s_2s_1^2$};
\node[main] (13) [left of=5] {$s_1s_2$};
\node[main] (14) [right of=6] {$s_2s_1s_1^2$};
\node[main] (15) [left of=7] {$s_1$};
\node[main] (16) [right of=8] {$s_2s_1^2$};

\draw[->] (1) -- (11);
\draw[->] (1) -- (4);
\draw[->] (1) -- (12);
\draw[->] (2) -- (3);
\draw[->] (2) -- (4);
\draw[->] (2) -- (12);
\draw[->] (3) -- (13);
\draw[->] (3) -- (5);
\draw[->] (3) -- (6);
\draw[->] (3) -- (14);
\draw[->] (4) -- (13);
\draw[->] (4) -- (6);
\draw[->] (11) -- (13);
\draw[->] (11) -- (5);
\draw[->] (11) -- (6);
\draw[->] (11) -- (14);
\draw[->] (12) -- (5);
\draw[->] (12) -- (14);
\draw[->] (5) -- (7);
\draw[->] (5) -- (8);
\draw[->] (5) -- (16);
\draw[->] (6) -- (15);
\draw[->] (6) -- (8);
\draw[->] (6) -- (16);
\draw[->] (13) -- (15);
\draw[->] (13) -- (8);
\draw[->] (13) -- (16);
\draw[->] (14) -- (7);
\draw[->] (14) -- (8);
\draw[->] (14) -- (16);
\draw[->] (7) -- (9);
\draw[->] (7) -- (10);
\draw[->] (8) -- (9);
\draw[->] (15) -- (9);
\draw[->] (15) -- (10);
\draw[->] (16) -- (10);
\end{tikzpicture}
\end{center}
where $u'\leq u$ if there is a directed path of arrows starting from $u$ and ending in $u'$.
    \qed
\end{example}

\section{Preliminaries}\label{sec:preliminaries}

In this section we introduce some definitions and results that will be used in this paper. For more details see, for instance, \cite{dkv, helgason, neeb, humphreys, knapp, w} (for Subsection \ref{subsec:lie-theory}), \cite{conley,fps,ps1} (for Subsection \ref{subsec:dynamics}) and \cite{psm1, psm2, patrao-santos, sm93, smt} (for Subsection \ref{subsec:controlsets}).

\subsection{Lie theory}\label{subsec:lie-theory}

\subsubsection{Homogeneous spaces}
\label{homogspaces}

For the theory of Lie groups and its homogeneous spaces we refer to Hilgert-Neeb \cite{neeb} and Knapp \cite{knapp}.  Let $G$ be a real Lie group with Lie algebra $\g$. Let a Lie group $G$ act on a manifold $F$ on the left by the differentiable map
$
G \times F \to F$,  $(g,x) \mapsto gx
$.
Fix a point $x \in F$.  The isotropy subgroup $G_{x}$ is the set of all $g \in G$ such that $gx = x$.  We say that the action is transitive or, equivalently, that $F$ is a homogeneous space of $G$, if $F$ equals the orbit $Gx$ of $x$ (and hence the orbit of every point of $F$).
In this case, the map
$$
G \to F, \qquad g \mapsto g x,
$$
is a submersion onto $F$ which is a differentiable locally trivial principal fiber bundle with structure group the isotropy subgroup $G_{x}$.
Quotienting by $G_{x}$ we get the diffeomorphism
$$
G/G_{x} \simto F, \qquad g G_{x} \mapsto g x.
$$

\subsubsection{Semi-simple Lie theory}\label{semisimple-lie-theory}

For the theory of real semisimple Lie groups and their flag manifolds we refer to Duistermat-Kolk-Varadarajan \cite{dkv}, Hilgert-Neeb \cite{neeb} and Knapp \cite{knapp}.
Let $G$ be a connected real Lie group with a semisimple Lie algebra $\g$, a finite center, and having a complexification.
Fix a Cartan decomposition $\g = \k \oplus \s$ and let $K$ be the connected subgroup with Lie algebra $\k$. We have that $K$ is a maximal compact subgroup of $G$.
Since $\ad(X)$ is skew-symmetric for $X \in \k$, the Cartan inner product is $K$-invariant.
Since $\ad(X)$ is symmetric for $X \in \s$, a maximal abelian subspace $\frak{a} \subset \frak{s}$ can be simultaneously diagonalized so that $\g$ splits as an orthogonal sum of
$$
\g_\alpha = \{ X \in \g:\, \ad(H)X = \alpha(H)X, \, \forall H \in \a \},
$$
where $\alpha \in \a^*$ (the dual of $\a$). We have that $\g_0 = \m \oplus \a$, where $\m$ is the centralizer of $\a$ in $\k$. A root is a functional $\alpha \neq 0$ such that its root space $\g_\alpha \neq 0$. Denoting the set of roots by $\Pi$, we thus have the root space decomposition of $\g$, given by the orthogonal sum
$$
\g = \m \oplus \a \oplus \sum_{\alpha \in \Pi} \g_\alpha.
$$
Fix a Weyl chamber $\frak{a}^{+}\subset \frak{a}$ and let $\Pi^{+}$ be the corresponding positive roots, $\Pi^- = - \Pi^+$ the negative roots and $\Sigma $ the set of simple roots.  Consider the nilpotent subalgebras
\[
\n^\pm = \sum_{\alpha \in \Pi^{\pm}}\frak{g}_{\alpha }
\]
such that
\[
\g = \m \oplus \a \oplus \n \oplus \n^-.
\]

In this paper, we look at $K$ as the homogeneous manifold $G/AN$, denoting by $b$ its base point $AN$, where $G=KAN$ is the associated Iwasawa decomposition of $G$. The natural action of $G$ on $G/AN$ is given by left multiplication (as in Section 10.1 of \cite{neeb}). From the Iwasawa decomposition, it follows that the map $K \to G/AN$ given by $k \mapsto k b$, is a $K$-equivariant diffeomorphism. The isotropy subalgebra of the base $b$ is given by $\a \oplus \n$. We also look at the maximal and the maximal extended flag manifolds of $G$, denoted, respectively, by $\F$ and $\F_0$, as the homogeneous manifolds $G/MAN$ and $G/M_0AN$, diffeomorphic, respectively, to $K/M$ and $K/M_0$, where $M$ is the centralizer of $\a$ in $K$ and $M_0$ denotes its connected component. We denote by $b_\F$ the base point $MAN$ and by $b_0$ the base point $M_0AN$.

The Weyl group $W$ is the finite group generated by the reflections over the root hyperplanes $\alpha=0$ in $\frak{a}$, $\alpha \in \Pi$. The Weyl group $W$ acts on $\frak{a}$ by isometries and alternatively can be described as $W=M_{*}/M$ where $M_{*}$ is the normalizer of $\a$ in $K$. An element $w$ of the Weyl group $W$ can act in $\g$ by taking a representative in $M_*$.  This action normalizes $\a$ and $\m$, permutes the roots $\Pi$ and thus permutes the root spaces $\g_\alpha$, where $w \g_\alpha = \g_{w \alpha}$ does not depend on the representative chosen in $M_*$.

Let $H \in \frak{a}$ and denote the centralizer of $H$ in $G$ and $K$ respectively by $G_{H}$ and $K_H$. We have that
\[
 G_{H} = G_{H}^0M
 \qquad \mbox{and} \qquad
 K_{H} = K_{H}^0M,
\]
where $G_{H}^0$ and $K_{H}^0$ are the connected components of the identity respectively of $G_{H}$ and $K_H$.
Consider the nilpotent subalgebras
\[
\frak{n}^\pm_H = \sum_{\pm \alpha(H) > 0} \g_\alpha,
\]
given by the the sum of the positive/negative eigenspaces of $\ad(H)$ in $\g$ and let $N_H^{\pm}$ be the corresponding connected Lie subgroups. Since $G_H$ leaves invariant each eigenspace of $\ad(H)$, it follows that $\frak{n}^\pm_H$ and $N_H^{\pm}$ are $G_H$-invariant.

For each $\alpha \in \Pi$, let $H_{\alpha} \in \mathfrak{a}$ such that $\langle H_{\alpha},H\rangle=\alpha(H)$, for all $H \in \mathfrak{a}$. Let $H \in \cl \mathfrak{a}^+$ and consider $\mathfrak{a}(H)$ the subalgebra of $\mathfrak{a}$ generated by $H_{\alpha}$ such that $\alpha \in \Sigma$ and $\alpha(H)=0$ and $\mathfrak{g}(H)$ the subalgebra of $\mathfrak{g}$ generated by
\[
\mathfrak{a}(H)\oplus \sum\mathfrak{g}_{\alpha}
\]
where the sum is taken over all $\alpha \in \Pi$ such that $\alpha(H)=0$. Let us also consider $\mathfrak{k}(H)=\mathfrak{k}\cap\mathfrak{g}(H)$ and
\[
\mathfrak{n}(H) = \sum\mathfrak{g}_{\alpha}
\]
where the sum is taken over all $\alpha \in \Pi^+$ such that $\alpha(H)=0$.
Let $G(H)$, $K(H)$ and $N(H)$
be the connected subgroups of $G$ generated by $\mathfrak{g}(H)$, $\mathfrak{k}(H)$ and $\mathfrak{n}(H)$,
respectively. We have that 
\[
N=N_HN(H), \quad N(H)=G_H \cap N, \quad K_H=K(H)M
\]
and
\[ 
G_H=K_HAN(H).
\]

\subsubsection{Canonical objects}\label{subsec:canonicalobjects}

One of the main differences between the theory of dynamics and semigroup actions on flag manifolds developed in \cite{dkv, fps, ps, smt} and the theory of dynamics and semigroup actions on maximal compact subgroups developed in \cite{ps1, patrao-santos} is the following. In flag manifolds, the objects are describe by the Weyl chambers which are in bijection with the classes of $G/MA$, while, in maximal compact subgroups, the objects are describe by the classes of $G/M_0A$, since the connected components of $M$ have a decisive role in these descriptions. 

Let $\g$ be a real semisimple Lie algebra, $\theta$ a Cartan involution of $\g$ and $\g=\k\oplus \s$ the Cartan decomposition of $\g$ defined by $\theta$. Given a maximal abelian subalgebra $\a$ of $\s$, we call $(\theta,\a)$ an \textit{admissible pair} of $\g$. Fixing a Weyl chamber $\a^+ \subset \a$, we call $(\theta,\a,\a^+)$ an \emph{admissible triple}. Let $G$ be a connected Lie group with finite center and Lie algebra $\g$, and $K$ be the connected subgroup of $G$ generated by $\exp(\k)$. An admissible triple $(\theta,\a,\a^+)$ determines the following objects:
$(1)$ the inner product $\langle \cdot, \cdot \rangle_{\theta}$, restricted to $\a$, 
$(2)$ the maximal abelian subgroup $A=\exp\a$, 
$(3)$ the normalizer $M_*A$ and the centralizer $MA$, of $\a$ in $G$, where $M_*$ and $M$ are the normalizer and the centralizer of $\a$ in $K$,
$(4)$ the quotient groups $U=M_*A/M_0A \simeq M_*/M_0$ \label{def:grupoU(B)} and $C=MA/M_0A\simeq M/M_0$, where $M_0$ is the identity component of $M$ and $M_*$, 
$(5)$ the Weyl group $W$, 
$(6)$ the system of roots $\Pi$,  
$(7)$ the simple system of roots $\Sigma$ and the set of positive roots $\Pi^+$, 
$(8)$ the minimal parabolic subalgebra $\p$ and the minimal parabolic subgroup $P$, and 
$(9)$ the components of the Iwasawa decompositions of $\p$ and $P$. Since $G$ has finite center, it follows that $K$, $M_*$ and $M$ are compact subgroups and $U$ and $C$ are finite groups. We have that $M$ is a normal subgroup of $M_*$ and $C$ is a normal subgroup of $U$. 
\textcolor{black}{Let us assume that $G$ has a complexification to guarantee that $C$ is an abelian group (see \cite{knapp}, Theorem 7.53). The above assumptions hold for all linear groups.}

The group $G$ acts on itself by conjugation, on its Lie algebra through adjoint action and on the dual of the algebra by coadjoint action. The following notation will be used for the adjoint action of $G$ on these objects
\begin{enumerate}[$(i)$]
    \item $gwg^{-1} = \Ad(g)w\Ad(g^{-1})$ for $w \in W$,
    \item $g\alpha = \Ad(g)^*\alpha$ for $\alpha \in \Pi$, and
    \item $gX = \Ad(g)X$ for $X \in \g$.
\end{enumerate}

The \emph{sets of maximal abelian subgroups and Weyl chambers} in $G$ are defined, respectively, by
\begin{equation*}
    \mathcal{A} = \{\exp(\a) : (\theta,\a) \,\, \mbox{admissible pair}\}
\end{equation*}
and
\begin{equation*}
    \mathcal{A}^+ = \{\exp(\a^+) : (\theta,\a,\a^+) \,\, \mbox{admissible triple}\}.
\end{equation*}
The group $G$ acts on $\mathcal{A}$ and $\mathcal{A}^+$ by conjugation in such a way that the sets $\mathcal{A}$ and $\mathcal{A}^+$ are identified with the homogeneous spaces $G/M_*A$ and $G/MA$, respectively. Since the subgroup $M_0A$ is the identity component of both $M_*A$ and $MA$ it follows that $M_0A$ is a normal subgroup of $M_*A$ and $MA$. Therefore, the canonical projections
\begin{equation}\label{eq:fibabcam}
    G/M_0A \to G/M_*A ~~~\text{and}~~~ G/M_0A \to G/MA,
\end{equation}
are principal bundles with structure groups $U$ and $C$, respectively. From the projections on (\ref{eq:fibabcam}), each coset $gM_0A$ determines the cosets $gM_*A$ and $gMA$ which, in turn, are identified respectively with the maximal abelian subgroup $gAg^{-1}=\exp g\a$ and with the Weyl chamber $gA^+g^{-1}=\exp g\a^+$. Consider the notation $B=M_0A$. The Weyl chamber determined by $gB$ will be denoted by $\la(gB)$, i.e., $\la(gB)=gA^+g^{-1}=\exp g\a^+$.

\begin{lemma}\label{lema:aplicacoeslambda}
    If $(\theta,\a,\a^+)$ and $(\o{\theta},\o{\a},\o{\a}^+)$ are admissible triple such that $\a^+=\o{\a}^+$, then they determine the same objects $(1)-(9)$ mentioned above. Thus, these objects are determined by each coset $gB$, $g \in G$, and will be denoted by the juxtaposition of $(gB)$ on the right hand side. In addition, given $g \in G$, we have
    \begin{enumerate}[$(i)$]
        \item\label{enum:objconjprimeiro} $\p(g B) = g\p(B)$ and the same for $\a$, $\m$, $\n$, $\Pi$ and $\Sigma$,
        \item $\langle gH, g\widetilde{H} \rangle({gB}) = \langle H, \widetilde{H} \rangle({B})$, for all $ H,      \widetilde{H} \in \a(B)$,
        \item $(M_*A)(gB) = g(M_*A)(B)g^{-1}$ and the same for $A$, $M$, $M_0$ and $N$, and
        \item\label{enum:objconjultimo} $U(gB)=gU(B)g^{-1}$ and the same for $C$ and $W$.
    \end{enumerate}
\end{lemma}

Defining
\begin{equation*}\label{eq:preU}
	\mathcal{U} = \{(g'B, u) : g' \in G, u \in U(g'B)\},
\end{equation*}
we have that the adjoint action of $G$ on $\mathcal{U}$, given by
\begin{equation*}
    g(g'B, u) = (g g'B, g u g^{-1})
\end{equation*}
where $g \in G$, is well defined since by Lemma \ref{lema:aplicacoeslambda}, $g u g^{-1} \in U(g g'B)$. The next result is Proposition 2.3 of \cite{patrao-santos}, which shows that the quotient of $\mathcal{U}$ by this action of $G$ is a group isomorphic to the group $U$ introduced in page \pageref{def:grupoU(B)}, called \emph{canonical extended Weyl group} or \emph{canonical Weyl-Tits group of $G$}. Similarly, canonical objects are constructed for other objects determined in Lemma \ref{lema:aplicacoeslambda} by cosets $gB$, $g \in G$. In Subsection \ref{subsec:controlsets} and Section \ref{sec:controlsets}, the canonical Weyl-Tits group will be denoted by $U$ while the group $U$ defined on page \pageref{def:grupoU(B)} will be denoted by $U(B)$, and the same for the other objects. An element of $U$ is an orbit of a pair $(g'B, u)$ denoted by $[(g'B, u)]$. 

\begin{proposition}\label{prop:Ucanonico}
    For each $g \in G$ and $u \in U$, there is a unique $u(gB) \in U(gB)$ such that $u = [(gB, u(gB))]$. For every $g \in G$, the map $u \mapsto u(gB)$ is a canonical isomorphism between the canonical group $U$ and the group $U(gB)$, where for all $u, \widetilde{u} \in U$, the product given by
    \begin{equation*}\label{eq:produtoU}
        u\widetilde{u} = [(gB,u(gB)\widetilde{u}(gB))]
    \end{equation*}
    and the inverse given by
    \begin{equation*}\label{eq:inversaU}
	u^{-1} = [(gB, u(gB)^{-1})]
    \end{equation*}
    are well defined. Furthermore, for all $g,g' \in G$, we have that
    \[
    u(gg'B) = g u(g'B) g^{-1}.
    \]
\end{proposition}

If $g \in G$ and $u \in U$, then 
\[
u(gB)gB=gu_*B
\]
where $u(B)=u_*(M_0A)(B)$, $u_*\in M_*$. 

Let $g \in G$, $\alpha \in \Pi(B)$ and $H \in \a(gB)$. From Lemma \ref{lema:aplicacoeslambda}, $g\alpha \in \Pi(gB)$, $g^{-1}H \in \a(B)$ and
\begin{equation*}
    \langle H,H_{g\alpha}\rangle(gB) = (g\alpha)(H)
	=\alpha(g^{-1}H) 
	= \langle g^{-1}H,H_{\alpha}\rangle(B)
	=\langle H,gH_{\alpha}\rangle(gB).
\end{equation*}
From this, we have that $H_{g\alpha}=gH_{\alpha}$ and thus $r_{g\alpha}=gr_{\alpha}g^{-1}$. Therefore, if $w \in W$ and $w(B)=r_{\alpha_1}\cdots r_{\alpha_d}$ is a decomposition of $w(B)$ as a product of reflections, then 
\[
w(gB)=gw(B)g^{-1}=(gr_{\alpha_1}g^{-1})\cdots (gr_{\alpha_d}g^{-1})=r_{g\alpha_1}\cdots r_{g\alpha_d}.
\]

The proof of the next proposition is analogous to that given for Lemma 2.2 and Proposition 2.3 of \cite{patrao-santos}.

\begin{proposition}\label{prop:orderW}
    Let $g \in G$ and $w, w' \in W$. Then
    \begin{enumerate}
        \item $w(B)=r_{\alpha_1}\cdots r_{\alpha_d}$ is a reduced expression of $w(B)$ if and only if $w(gB)=r_{g\alpha_1}\cdots r_{g\alpha_d}$ is a reduced expression of $w(gB)$.
        \item $w'(B) \leq w(B)$ if and only if $w'(gB)\leq w(gB)$.
        \item The relation $\leq$ on the canonical Weyl group $W$, defined by $w'\leq w$ if and only if $w'(gB) \leq w(gB)$, is the Bruhat-Chevalley order.
    \end{enumerate}
\end{proposition}

\subsection{Dynamics}\label{subsec:dynamics}

We recall some concepts of topological dynamics (for more
details, see \cite{conley}). Let $\phi :\mathbb{T}\times X \to X$
be a continuous dynamical system on a compact metric space $(X, d)$, with
discrete $\mathbb{T}={\mathbb Z}$ or continuous
$\mathbb{T}={\mathbb R}$ time.
Denote by $\omega(x)$, $\omega^*(x)$, respectively, the forward and
backward omega limit sets of $x$. A Morse decomposition of $\phi^t$ is given by a finite collection of disjoint subsets $
\{{\mathcal M}_{1},\ldots ,{\mathcal M}_{n}\}$ of $X$ such that
\begin{enumerate}[$(i)$]
\item each ${\mathcal M}_i$ is compact and
$\phi\,^t$-invariant,

\item for all $x \in X$ we have $\omega(x),\, \omega^*(x)
\subset \bigcup_i {\mathcal M}_i$,

\item if $\omega(x),\, \omega^*(x) \subset {\mathcal M}_j$
then $x \in {\mathcal M}_j$.
\end{enumerate}
Each element of a decomposition is called a {\em Morse component} and the unstable (or stable) set of a Morse component ${\mathcal M}_i$ is the set ${\rm un}({\mathcal M}_i)$ (or ${\rm st}({\mathcal M}_i)$) of all points whose backward (or forward) omega limit set is contained in ${\mathcal M}_i$. The order between Morse components is given by ${\mathcal M}_i \leq {\mathcal M}_j$ if and only if ${\cl(\rm un}({\mathcal M}_j)) \subset \cl({\rm un}({\mathcal M}_i)$) (or $\cl({\rm st}({\mathcal M}_i)) \subset \cl({\rm st}({\mathcal M}_j))$). The {\em minimal} Morse decomposition is a Morse decomposition which is contained in every other Morse decomposition. Each set ${\mathcal M}_i$ of a minimal Morse decomposition is called a {\em minimal Morse component}. 

In the following theorems, we collect some previous results about the dynamics of a flow $g^t$ of translations of a real semisimple Lie group $G$ acting on its maximal flag manifold $\F \simeq G/MAN $, on its maximal compact subgroup $K \simeq G/AN $, and on its maximal extended flag manifold $\F_0 \simeq G/M_0AN $, where $t \in \mathbb{T}$ with $\mathbb{T} = \Z$ or $\mathbb{T} = \R$. The flow $g^t$ is either given by the iteration of some $g \in G$, when $\mathbb{T} = \Z$, or by $\exp(tX)$, when $\mathbb{T} = \R$, where $X \in \g$, and $g^t$ acts on $\F$, $K$, and $\F_0$ by left translations. 

The usual additive Jordan decomposition writes a matrix as a commuting sum of a semisimple and a nilpotent matrix and we can further decompose the semisimple part as the commuting sum of its imaginary and its real part, where each part commutes with the nilpotent part and the matrix is diagonalizable over the complex numbers if and only if its nilpotent part is zero.
This generalizes to a multiplicative Jordan decomposition of
the flow $g^t$ in the semisimple Lie group $G$ (see Section 2.3 of \cite{fps}), providing us with a commutative decomposition
\[
g^t = e^t h^t u^t.
\]
There exist a Cartan decomposition of $\g$ with a corresponding maximal compact subgroup $K$ and a Weyl chamber $\a^+$ such that the elliptic component $e^t$ lies in $K$, the hyperbolic component is such that $h^t = \exp(tH)$, where $H \in \cl \a^+$, and the unipotent component is such that $u^t = \exp(tN)$, with $N \in \g$ nilpotent. Furthermore, we have that $h^t$, $e^t$ and $u^t$ lie in $G_H$, the centralizer of $H$ in $G$. Finally the hyperbolic component $H$ dictates the minimal Morse components (see Proposition 5.1 and Theorem 5.2 of \cite{fps}).
First we state the dynamical results involving $\F$ presented in \cite{fps}.

\begin{theorem}
Let $g^t$ be a general translation on $\F \simeq G/MAN$ and $g^t = e^t h^t u^t$ be its Jordan decomposition. Then
\begin{enumerate}[$(i)$]
\item The minimal Morse components of $g^t$ are given by
\[
 {\cal M}(g^t,w) = G^0_Hwb_\F = K^0_Hwb_\F.
\]

\item The minimal Morse components of $g^t$ are in bijection with the elements of $W_H \backslash W$, where $W_H$ is the centralizer of $H$ in $W$.

\item The unstable manifold of ${\cal M}(g^t,w)$ is given by
\[
 {\rm un}(g^t,w) = N_H {\cal M}(g^t,w).
\]

\item When $g$ is a regular element of $G$, then
\[
{\cal M}(g^t,w) = wb_\F
\qquad
\mbox{and}
\qquad
{\rm un}(g^t,w) = Nwb_\F =: \B(w)
\]
is a Bruhat cell.
\end{enumerate}
\end{theorem}

Now we state the dynamical results involving $K$ presented in \cite{ps1} and the results involving $\F_0$. Note that the results involving $\F_0$ follow immediately from the results in $K$, since $M_0 \subset K_H^0$.

\begin{theorem}
Let $g^t$ be a general translation on $K \simeq G/AN$ and $g^t = e^t h^t u^t$ be its Jordan decomposition. Then
\begin{enumerate}[$(i)$]
\item The minimal Morse components of $g^t$ are given by
\[
 {\cal M}(g^t,u) = G^0_Hub \cup gG^0_Hub,
\]
where
\[
G^0_Hub = K^0_Hub.
\]
If $g$ is in the connected component of the identity $G_H^0$ of the centralizer of $H$ in $G$, which always happens when $\mathbb{T} = \R$, then ${\cal M}(g^t,u)$ is connected and equal to $G^0_Hub$. If $g \notin G_H^0$, then ${\cal M}(g^t,u)$ has two connected components $G^0_Hub$ and $gG^0_Hub$.

\item The minimal Morse components of $g^t$ are in bijection with the elements of $U_H^g \backslash U$, where
\[
U_H^g := U_H \cup c_gU_H,
\]
for an element $c_g \in C$ such that $G^0_Hc_gub = gG^0_Hub$, and 
\[
U_H := \{u \in U: u \subset K^0_H\}.
\]
The attractors are given by ${\cal M}(g^t,c)$ where $c \in C$.

\item The unstable manifold of ${\cal M}(g^t,u)$ is given by
\[
 {\rm un}(g^t,u) = N_H {\cal M}(g^t,u).
\]
\end{enumerate}
The same results are valid for a general translation $g^t$ on $\F_0 \simeq G/M_0AN$ just replacing $b$ by $b_0$, ${\cal M}$ by ${\cal M}_0$, and ${\rm un}$ by ${\rm un}_0$. In this case, when $g$ is a regular element of $G$, then
\[
{\cal M}_0(g^t,u) = ub_0
\qquad
\mbox{and}
\qquad
{\rm un}_0(g^t,u) = Nub_0 =: \B(u)
\]
is a Bruhat cell.
\end{theorem}

\subsection{Semigroup actions}\label{subsec:controlsets}

Let $X$ be a topological space. A \emph{local map} on $X$ is a continuous map $\phi:\mathrm{dom}\phi \to X$, where $\mathrm{dom }\phi \subset X$ is an open set. A \emph{local semigroup on $X$} is a family $S$ of local maps that is closed to the composition of maps, in the following sense: if $\phi, \psi \in S$ and $\phi^{-1}(\mathrm{dom}\psi) \neq \emptyset$ then the composition
\[
\psi\circ \phi:\phi^{-1}(\mathrm{dom}\psi) \to X
\]
also belongs to $S$. A local semigroup $S$ on $X$ acts on $X$ by evaluation of maps. For $x \in X$, we denote its
orbit and its backward orbit, respectively, by $Sx$ and $S^*x$, i.e.,
\[
Sx=\{\phi(x):\phi \in S, x \in \mathrm{dom}\phi\} ~\textrm{ and }~ S^*x=\{y: \exists \phi\in S, \phi(y)=x\}.
\]
The action of $S$ on $X$ induces transitive relations $\preceq$, $\preceq_w$ and $\preceq_s$ on $X$, called respectively \emph{algebraic, weak and strong relations} and defined in the following way. Given $x, y \in X$,
\begin{enumerate}
    \item $x \preceq y$ if and only if $y\in Sx$.
    \item $x \preceq_w y$ if and only if $y\in \mathrm{cl}(Sx)$.
    \item $x \preceq_s y$ if and only if $x\in \mathrm{int}(S^{*}y)$.
\end{enumerate}
In general, given a relation $\leq$ on a set, its \emph{symmetrization} $\sim$ is defined by $x\sim y$ if and only if $x\leq y$ and $y\leq x$. We denote by $\sim$, $\sim_w$ and $\sim_s$ the symmetrizations of $\preceq$, $\preceq_w$ and $\preceq_s$, respectively. It is clear that $\sim$, $\sim_w$ and $\sim_s$ are transitive and symmetric but may fail to be reflexive. 
We also denote by $[x]$, $[x]_w$ and $[x]_s$ the classes of $x$ with respect to $\sim$, $\sim_w$ and $\sim_s$, respectively. We have $[x]_s\subset [x]\subset [x]_w$, for all $x \in X$. It is not difficult to see that $[x]_s \neq \emptyset$ if and only if $x\sim_sx$. In this case, $x$ is said to be \emph{$S$-self-accessible}.

A \emph{control set of $S$} is a  weak class $[x]_w$ of an element $x\in X$ that is $S$-self-accessible. Let $D=[x]_w\subset X$ be a control set of $S$. If $y \in D$ is self-accessible, then $[y]_s=[x]_s$ (see \cite{psm1}, Corollary 4.7), i.e., in each control set there is a unique strong class, called the \emph{transitivity set} of $D$ and denoted by $D_0$. We have
\[
D_0=\{y \in D: y \text{ is self-accessible}\}.
\]
A control set $D$ is called \emph{$S$-invariant} if $Sx\subset D$, for all $x \in D$. Traditionally, in the context of a semigroup $S$ with nonempty interior of a Lie group $G$, a control set $D$ is a weak class such that there exist $x \in D$ and $g$ in the interior of $S$ with $gx = x$, while its transitivity set is given 
\[
D_0 = \{x \in D : gx = x \mbox{ for some } g \mbox{ in the interior of } S \}
\]
and it can be proven that $D_0$ is dense in $D$ (see \cite{smm}). Since this traditional definition of transitivity set depends only on the interior of $S$, without lost of generality, we can assume that we are dealing with open semigroups of Lie groups. And for open semigropus, it follows that the traditional definition of transitivity set coincides with the one previously presented. 

The \emph{order between control sets} is defined in such way that, if $D$ and $D'$ are control sets, then $D\leq D'$ if, and only if, $x\preceq x'$, for some $x\in D_0$ and $x' \in D'_0$. It is well known (see e.g. \cite{psm1}, Proposition 4.11) that $D\leq D'$ if and only if $x\preceq x'$ for all $x \in D_0$ and $x' \in D'_0$.

\subsubsection{Control sets on maximal flag manifolds}

The \emph{maximal flag manifold $\Bbb{F}$ of $\g$} is the set
\begin{equation*}
    \Bbb{F} = \{\p(gB) : g \in G\}
\end{equation*}
of all minimal parabolic subalgebras of $\g$, which is diffeomorphic to the homogeneous space $G/P$. The adjoint action of
$G$ on $\Bbb{F}$, given by $g'\p(gB)$, is well defined, because, by Lemma \ref{lema:aplicacoeslambda}, we have that
\begin{equation*}\label{eqacaoGemFT}
    g'\p(gB) = \p(g' gB).
\end{equation*}

Let $g \in G$. For every $h$ in the Weyl chamber $\la(gB)$, the set of fixed points of the action of $h$ on $\F$ is given by
\[
\{w(gB)\p(gB):w \in W\}.
\]
Furthermore, $\p(gB)$ is the only attractor of $h$. For each $w\in W$ the fixed point $w(gB)\p(gB)$ is denoted by $\p_w(gB)$ and called \emph{fixed point of type $w$ in $\F$ determined by $gB$}. For every $g \in G$ we have
\begin{equation*}\label{eq:gpw}
    g\p_w(B)=\p_w(gB).
\end{equation*}

Let $S\subset G$ be an open semigroup. The semigroup $S$ acts on $\F$ by restricting the action of $G$. The \emph{sets of attractors $\mathrm{atr}(S)$ and fixed points $\mathrm{fix}_w(S)$ of type $w$ of $S$}, in $\F$, are defined by
\[
\mathrm{atr}(S)=\{\p(gB): \la(gB) \cap S \neq \emptyset\}
\]
and
\[
\mathrm{fix}_w(S)=\{\p_w(gB): \la(gB) \cap S \neq \emptyset\}.
\]
In particular, $\mathrm{fix}_1(S)=\mathrm{atr}(S)$.
The following result characterizes the control sets of $S$ on the maximal flag manifold $\F$ of $\g$ (see \cite{sm93, smt}).

\begin{theorem}\label{teo:controleweyl}
    For every $w \in W$, there exists a control set $\Bbb{C}(w)$ on $\F$ such that
    \begin{equation*}
        \Bbb{C}(w)_0 = \mathrm{fix}_w(S)
    \end{equation*}
    and these are all control sets of $S$ on $\Bbb{F}$. There is a unique invariant control set $\Bbb{C} = \Bbb{C}(1)$ and $\Bbb{C}_0 = \mathrm{atr}(S)$. Furthermore, the set
    \[
    W(S)=\{w \in W:\Bbb{C}(w)=\Bbb{C}\}
    \]
    is a subgroup of $W$.
\end{theorem}

\subsubsection{Control sets on maximal compact subgroups}\label{sec:controlsetK}

In this subsection we collect some results about control sets on $\Bbb{F}_0\simeq K/M_0$. Note that the results for $\Bbb{F}_0$ follow directly from the results for $K$, presented in \cite{patrao-santos}. The first proposition gives a relationship between control sets of $S$ on $K/M_0$ with the control sets of $S$ on the maximal flag manifold $K/M$ (see \cite{patrao-santos}, Propositions 3.4 and 3.14 and Theorem 3.6).

\begin{proposition}\label{prop:S-ics}
    Let $w \in W$ and $\Bbb{D}\subset \pi^{-1}(\Bbb{C}(w))$ be a control set on $K/M_0$.
    \begin{enumerate}
        \item $\Bbb{D}$ is $S$-invariant if and only if $w \in W(S)$.
        \item If $\Bbb{D}'\subset \pi^{-1}(\Bbb{C}(w))$ is a control set, then there exists $c \in C$ such that $\Bbb{D}'=\Bbb{D}c$.
    \end{enumerate}
\end{proposition}

In order to relate the control sets of $S$ to their fixed points, we give below some notations and results. In this sense, the next proposition (see \cite{patrao-santos}, Proposition 3.9 and \cite{ps1}, Theorem 3.9) describes the set of fixed points and the attractors of elements of a Weyl chamber.

\begin{proposition}\label{prop:fixtipou}
    Let $g \in G$. For every $h$ in the Weyl chamber $\la(gB)$, the set of fixed points of the action of $h$ on $K/M_0$ is given by
    $
    \{u(gB)gb_0:u \in U\}
    $
    and the attractors are given by $c(gB)gb_0$, $c \in C$. Furthermore, if $h \in \la(gB)$ and $n \in N(gB)$, then $h^{-k}nh^k \to 1$ when $k \to \infty$.
\end{proposition}

This proposition leads to the following definition.

\begin{definition}\label{def:pontofixou}
    For each $u \in U$ and $g \in G$, define the \emph{fixed point of the type $u$} in $\Bbb{F}_0\simeq G/M_0AN$ determined by $gB$ as
    \begin{equation*}\label{eq:bu}
        b_u(gB) = u(gB)gb_0 = gu_*b_0
    \end{equation*}
    where $u(B)=u_*(M_0A)(B)$.
\end{definition}

The element $b_1(gB)$ is denoted by $b(gB)$, where $1$ is the identity of $U$. The following result is part of Lemma 3.12 of \cite{patrao-santos}.

\begin{lemma}\label{lema:fixtipou}
    Let $g, g' \in G$. For all $u, u' \in U$, we have
    \begin{enumerate}
	\item $b_u(gg'B) = gb_u(g'B)$, and
	\item $b_{u'}(u(gB)gB) = b_{uu'}(gB)$.
    \end{enumerate}
    Furthermore, $b(g'B)=b(gB)$ if and only if $g'B=ngB$, for some $n \in N(gB)$.
\end{lemma}

Let $\la_0$ be a fixed Weyl chamber in $G$ such that $S \cap \la_0 \neq \emptyset$. Let $\Bbb{D}$ be the invariant control set such that
\[
b(B)=b_0 \in \Bbb{D}_0,
\]
by Proposition 3.19 of \cite{patrao-santos}.
The \emph{set of cosets determined by $S$ and $\Bbb{D}$, $\mathcal{B}(S)$}, and the set of \emph{fixed points $\mathrm{fix}_u(S)$ of $S$ of type $u\in U$}, in $\Bbb{F}_0$, are defined by
\[
\mathcal{B}(S)=\{gB:\la(gB) \cap S \neq \emptyset \quad\mbox{and}\quad b(gB) \in \Bbb{D}_0\},
\]
and
\begin{equation*}\label{eq:fixSuc}
    \mathrm{fix}_u(S)=\{b_{u}(gB): gB \in \mathcal{B}(S)\}.
\end{equation*}

The following result characterizes the control sets of $S$ on $\Bbb{F}_0$ in terms of their fixed points (see \cite{patrao-santos},  Theorems 3.21 and 3.23).

\begin{theorem}\label{teo:controleU}
    For every $u \in U$, there exists a control set $\Bbb{D}(u)$ on $\Bbb{F}_0$ such that
    \[
    \Bbb{D}(u)_0=\mathrm{fix}_u(S)
    \]
    and these are all control sets of $S$ on $\Bbb{F}_0$. Moreover, if $c \in C$, then
    \begin{enumerate}
	\item $\Bbb{D}(uc)=\Bbb{D}(u)c$ ~and~ $\Bbb{D}(cu)=\Bbb{D}(u)u^{-1}cu$.
	\item $\Bbb{D}(c)$ is an invariant control set and ~$\Bbb{D}(1)=\Bbb{D}$.
	\item $\Bbb{D}c=\Bbb{D}$ if and only if $\Bbb{D}(cu)=\Bbb{D}(u)$.
    \item $\pi(\textcolor{black}{\Bbb{D}(u)_0})=\Bbb{C}(\pi(u))_0$.
    \end{enumerate}
\end{theorem}

\section{Order of minimal Morse components}

First we note that the natural projection from $K$ to $K/M_0$ establishes a bijection between the Bruhat cells of $K$ and the Bruhat cells of $K/M_0$ and thus establishes a bijection between the Schubert cells of $K$ and the Schubert cells of $K/M_0$. Hence the order of Morse components of a regular element in $K$ coincides with the order of Morse components of this element in $K/M_0$. In order to determine this order, we start constructing a map $\Psi_u$ from a closed ball to the Schubert cell $\S(u)$ in $K/M_0$ for each $u \in U$ following similar steps used in \cite{ps2}, for the case where the Lie algebra of $G$ is a split real form and thus $K \sim K/M_0$, and in \cite{lonardo}, for the flag manifolds.

For each $\alpha \in \Pi$, we write $\mathfrak{g}(\alpha)$ for the semisimple Lie subalgebra generated by $\mathfrak{g}_{\pm \alpha}$. Let $G(\alpha)$ be the connected Lie subgroup with Lie algebra $\mathfrak{g}(\alpha)$, $K(\alpha)\subset G(\alpha)$ a maximal compact subgroup and $M(\alpha)$ the centralizer of $\mathfrak{a}(\alpha)$ in $K(\alpha)$. Since $G(\alpha)$ has real rank one, its Weyl group is of the form $W=\{1,w\}$ and its flag manifold $\mathbb{F}_\alpha$ is the sphere $\mathbb{S}^m$, $m=\dim \mathfrak{s}({\alpha})-1$, where $\mathfrak{s}(\alpha)$ is the symmetric part of a Cartan decomposition of $\mathfrak{g}(\alpha)$. The origin of $\mathbb{F}_\alpha$ is ${b}_{\alpha}=e_1$, the first vector of the canonical basis of $\R^{m+1}$. Note that if $m>1$, then $M(\alpha)$ is a connected subgroup (see \cite{knapp}, Theorem 7.66). In the following, we restate Lemma 1.6 of \cite{lonardo} and provide a correct proof of it.

\begin{lemma}\label{lemma:psiposto1}
    If $B^m$ is the closed ball in $\R^m$, then there exists a continuous map $\psi:B^m \to K(\alpha)$ such that
    \begin{itemize}
        \item $\psi(\mathbb{S}^{m-1})\subset M(\alpha)$.
        \item The map $B^m\backslash\mathbb{S}^{m-1} \to \mathbb{S}^{m}\backslash\{{b}_{\alpha}\}$, $x \mapsto \psi(x){b}_{\alpha}$, is a diffeomorphism onto the subset ~$\mathbb{S}^{m}\backslash\{{b}_{\alpha}\}$.
    \end{itemize}
\end{lemma}

\begin{proof}
As in \cite{lonardo}, we need to consider only the cases $G(\alpha)=\text{SO}(1,n)$, $\text{SU}(1,n)$ and $\text{Sp}(1,n)$, which contain the compact subgroups SO$(n)$, SU$(n)$ and Sp$(n)$.
Let  $z=0$ and $v \in \R^{n-1}$, $z \in \text{Im} \mathbb{C}$ and $v\in \mathbb{C}^{n-1}$ or $z \in \text{Im}\mathbb{H}$ and $v \in \mathbb{H}^{n-1}$, with $v\neq 0$ and $|z|^2+|v|^2=1$, depending on whether $G(\alpha)=\text{SO}(1,n)$, $\text{SU}(1,n)$ or $\text{Sp}(1,n)$, respectively.
Let us consider the $n\times n$ matrices 
\[
A=A{(z,v)}=
\left(
\begin{array}{ccc}
	z & -\overline{v}^t \\
	v & -\frac{vz\overline{v}^t}{|v|^2} \\
\end{array}
\right) 
\mbox{ and }
J=J_{v}=
\left(
\begin{array}{ccc}
	1 & 0 \\
	0 & \frac{v\overline{v}^t}{|v|^2} \\
\end{array}
\right).
\]
We have
\[
A^2=
\left(
\begin{array}{cc}
	z & -\overline{v}^t \\
	v & -\frac{vz\overline{v}^t}{|v|^2} \\
\end{array}
\right)^2
=
\left(
\begin{array}{cc}
	-|z|^2-|v|^2 & 0 \\
	0 & -v\overline{v}^t + \frac{vz\overline{v}^t}{|v|^2}\frac{vz\overline{v}^t}{|v|^2} \\
\end{array}
\right)
=-J,
\]
since $z^2=-|z|^2$, ~$|z|^2+|v|^2=1$ ~and~ 
\[
-v\overline{v}^t+\frac{vz\overline{v}^t}{|v|^2}\frac{vz\overline{v}^t}{|v|^2}
=-v\overline{v}^t-|z|^2\frac{v\overline{v}^t}{|v|^2}
=-v\overline{v}^t-(1-|v|^2)\frac{v\overline{v}^t}{|v|^2}
=-\frac{v\overline{v}^t}{|v|^2}.
\]
Furthermore
\[
A^3=-JA=-A,
\]
since $JA=A$. Thus,
\[
A^n=
\left\{
\begin{array}{cll}
	(-1)^kJ,  & if & n=2k \\
	(-1)^kA,  & if & n=2k+1\\
\end{array}
\right.
\]
and, therefore,
\begin{eqnarray*}
	e^{tA}
	&=& \sum_{n=0}^{\infty}\frac{t^n}{n!}A^n = \sum_{k=0}^{\infty}\frac{t^{2k}}{(2k)!}A^{2k} + \sum_{k=0}^{\infty}\frac{t^{2k+1}}{(2k+1)!}A^{2k+1}\\
	&=& I-J+\sum_{k=0}^{\infty}(-1)^k\frac{t^{2k}}{(2k)!}J + \sum_{k=0}^{\infty}(-1)^k\frac{t^{2k+1}}{(2k+1)!}A\\
	&=&I-J+\cos(t)J + \sin(t)A,
\end{eqnarray*}
where $I$ is the $n\times n$ identity matrix. 

Note that
\[	
A+\overline{A}^t=0 ~~\textrm{if and only if}~~ \overline{z}=-z ~~\textrm{if and only if}~~ z=0, ~z \in \mathrm{Im}\Bbb{C} ~\mbox{ or }~ z \in \mathrm{Im}\mathbb{H}.
\]
Furthermore, since
\begin{itemize}
	\item $\displaystyle \frac{vz\overline{v}^t}{|v|^2} v = \frac{vz|v|^2}{|v|^2}=vz$, and
	\item $\displaystyle \frac{vz\overline{v}^t}{|v|^2} u =0$, for all $u$ such that $u\perp v$,
\end{itemize}
we have 
\begin{itemize}
	\item $z$ is eigenvalue of $\displaystyle \frac{vz\overline{v}^t}{|v|^2}$ of multiplicity 1, and
	\item $0$ is eigenvalue of $\displaystyle \frac{vz\overline{v}^t}{|v|^2}$ of multiplicity $n-1$.
\end{itemize}
Thus, $\displaystyle \mathrm{tr}\left(\frac{vz\overline{v}^t}{|v|^2}\right) = z$ and, therefore
$\mathrm{tr}(A)=z-\mathrm{tr}\left(\frac{vz\overline{v}^t}{|v|^2}\right)=0$.

Let $m=n$, $m=2n$ or $m=4n$ and $K=\text{SO}(n)$, $\text{SU}(n)$ or $\text{Sp}(n)$, depending on whether $G(\alpha)=\text{SO}(1,n)$, $\text{SU}(1,n)$ or $\text{Sp}(1,n)$, respectively. Consider the map $\psi:[0,\pi]\times \mathbb{S}^{m-1} \to K$, given by
\[
\psi(t,(z,v))=e^{tA{(z,v)}}w=(I-J_v)w+\cos(t)J_vw + \sin(t)A(z,v)w,
\]
where $w$ is the nontrivial element of the Weyl group of $G(\alpha)$. We have that $\psi$ is a continuous map and since, in any case, $w$ can be represented by the $n\times n$ matrix $w=\mathrm{diag}(-1,-1,1, \ldots, 1)$ it follows that
\[
\psi(0,(z,v))=w ~~ \mathrm{ and }~~ \psi(\pi,(z,v))=(I-2J_v)w \in M(\alpha),
\]
since the $(1,1)$-entry of $\psi(\pi,(z,v))$ is equal to 1. Thus, $\psi(\{\pi\}\times \mathbb{S}^{m-1}) \subset M(\alpha)$. Furthermore,
\begin{eqnarray*}
	\psi(t,(z,v)){b}_{\alpha}= -\cos(t)
	\left(
	\begin{array}{c}
		1 \\
		0 \\
	\end{array}
	\right)
	-\sin(t)
	\left(
	\begin{array}{c}
		z \\
		v \\
	\end{array}
	\right)
\end{eqnarray*}
and the map
\[
[0,\pi)\times \mathbb{S}^{m-1} \to \mathbb{S}^{m}\backslash\{{b}_{\alpha}\}, ~(t,(z,v))\mapsto \psi(t,(z,v)){b}_{\alpha}
\]
is a diffeomorphism onto $\mathbb{S}^{m}\backslash\{{b}_{\alpha}\}$. Therefore, $\psi$ satisfies the conditions of the lemma.
\end{proof}

Let us consider the identification $\Bbb{F}_0 \simeq K/M_0\simeq G/M_0AN$ and denote by $b_0=M_0AN$ the origin of $\Bbb{F}_0$.
Thus, for all $u\in U$, the Bruhat cell determined by $u$, in $\Bbb{F}_0$, is given by 
\[ 
\B(u)=Nub_0=Nu_*b_0, 
\] 
where $u_*\in M_*$ is a representative of $u$ in $M_*$. 
For each $\alpha \in \Pi$, we can choose a normalized $E_\alpha \in \g_{\alpha}$ such that $\langle E_\alpha, \theta E_\alpha \rangle = -2\pi^2/\langle H_\alpha, H_\alpha \rangle$. Hence, there is an isomorphism from $\sl(2,\R)$ to $\R E_{\alpha} \oplus \R H_{\alpha} \oplus \R \theta E_{\alpha}$ that can be complexified taking $\sl(2,\mathbb{C})$ to $\mathbb{C} E_{\alpha} \oplus \mathbb{C} H_{\alpha} \oplus \mathbb{C} \theta E_{\alpha}$.

\begin{definition}
Let
\[
s_{\alpha} := \exp(F_{\alpha}/2) \in U,
\]
where $F_{\alpha} := E_{\alpha} + \theta E_{\alpha} \in \k$ for a fixed chosen $E_{\alpha} \in \g_{\alpha}$. Let us denote $F_j = F_{\alpha_j}$ and $s_j = s_{\alpha_j}$, so that $s_j^2 = \exp (F_j) \in C$.
\end{definition}

If $\dim \g_\alpha > 1$, then $s^2_\alpha =1$, since $M(\alpha)$ is a connected subgroup.

\begin{lemma}\label{lemma:psiposto1dim1}[\cite{lonardo}, Lemma 1.7]
    The one-dimensional version of Lemma \ref{lemma:psiposto1} is realized by
    \[
    \psi:[0,1] \to K(\alpha), ~t \mapsto \exp(tF_{\alpha}).
    \]
    In particular, $\psi(0)=1$, $\psi(1/2)=\exp(F_{\alpha}/2)=s_{\alpha}$ and $\psi(1)=\exp(F_{\alpha})=s^2_{\alpha}$.
\end{lemma}

Let $u = s_1c_1 \cdots s_dc_d \in U$, where $c_i \in C$. Note that $\pi(u)=\pi(s_1c_1 \cdots s_dc_d) = r_1 \cdots r_d$, since $\pi(s_i) =r_i$ and $\pi(c_i)=1$, where $r_i=r_{\alpha_i}$.
For each $i\in \{1, \ldots, d\}$, let $\psi_i : B^{m_i} \to K({\alpha_i})$ be a map that satisfies Lemma \ref{lemma:psiposto1}, if $m_i>1$, or Lemma \ref{lemma:psiposto1dim1}, if $m_i=1$, where $m_i$ is the dimension of the flag of $G(\alpha_i)$.

\begin{definition}\label{def:char-function}
Let $u = s_1c_1 \cdots s_dc_d \in U$, where $c_i \in C$, such that $w = \pi (u) = r_1 \cdots r_d$ is a reduced expression. 
Let $B^m=B^{m_1}\times \cdots \times B^{m_d}\subset \R^m$ be the closed ball with dimension $m=m_1 +\cdots + m_d$. 
Then define the maps $\Psi_w:B^m \to \F$ and $\Psi_u : B^m \to \Bbb{F}_0$, given by
\begin{eqnarray*}
    \Psi_w (t_1, \dots, t_d)&:=& \psi_1 (t_1)c_1 \cdots \psi_d (t_d)c_d \cdot b_\F\qquad \mbox{and}\\
    \Psi_u (t_1, \dots, t_d)&:=& \psi_1 (t_1)c_1 \cdots \psi_d (t_d)c_d \cdot b_0.
\end{eqnarray*}
Choosing $\widetilde{t}_i \in \textnormal{Int}B^{m_i}$, for each $i$, such that  $\psi_i(\widetilde{t}_i)=s_i$, ($\widetilde{t}_i=0$ if $m_i>1$ and $\widetilde{t}_i=1/2$ if $m_i=1$) we have
\[
u = s_1c_1 \cdots s_d c_d = \psi_1 (\widetilde{t}_1)c_1 \cdots \psi_d (\widetilde{t}_d) c_d,
\]
so that
\[
ub_0=\Psi_u (\widetilde{t}_1, \dots, \widetilde{t}_d) \in \Psi_u (\textnormal{Int}B^m).
\]
\end{definition}

We note that $\mathbb{S}^{m-1}=\partial B^m =\partial (\text{Int}B^m)$, where $\partial X$ and $\text{Int}X$ are the frontier and the interior of $X$, respectively. Furthermore, the map $\Psi_w$ in Definition \ref{def:char-function}, satisfies the requirement of the Proposition 1.9 of \cite{lonardo}. In fact,
\begin{eqnarray*}
    \Psi_w (t_1, \dots, t_d)
    &=& \psi_1 (t_1)c_1\psi_2 (t_2)c_2 \cdots \psi_d (t_d)c_d \cdot b_\F\\
    &=&\psi_1 (t_1) \left(c_1 \psi_2(t_2)c_1^{-1}\right)\left(c_1c_2 \psi_3(t_3)c_2^{-1}c_1^{-1}\right)c_1c_2c_3 \cdots \psi_d(t_d)c_d \cdot b_\F\\
    &=& \prod_{i=1}^{d}\left(c_1\cdots c_{i-1}\psi_i (t_i)c_{i-1}^{-1}\cdots c_{1}^{-1}\right) \widetilde{c} \cdot b_\F\\
    &=& \prod_{i=1}^{d}\left(c_1\cdots c_{i-1}\psi_i (t_i)c_{i-1}^{-1}\cdots c_{1}^{-1}\right) \cdot b_\F,
\end{eqnarray*}
where $c_0=1$, $\widetilde{c}=c_1\cdots c_d \in C$, $c_1\cdots c_{i-1}\psi_i (t_i)c_{i-1}^{-1}\cdots c_{1}^{-1} \in K(\alpha_i)$ and $\psi_i (t_i) \in M(\alpha_i)$ if and only if $c_1\cdots c_{i-1}\psi_i (t_i)c_{i-1}^{-1}\cdots c_{1}^{-1} \in M(\alpha_i)$.

Now we obtain a similar result to Proposition 4.3 of \cite{ps2} and Proposition 1.9 of \cite{lonardo}.

\begin{proposition} \label{difeo from cube to cell}
Let $u = s_1c_1 \cdots s_dc_d \in U$, where $c_i \in C$, such that $w = \pi (u) = r_1 \cdots r_d$ is a reduced expression. 
Then
\[
\S (u) = \textnormal{cl} (\B (u)) = \Psi_u (B^m).
\]
The frontier of $\B (u)$ is
\[
\Psi _u \left( \partial B^m \right) = \partial \B (u) = \S (u) \backslash \B (u)
\]
and $\Psi_u |_{\textnormal{Int}B^m}$ is a diffeomorphism from $\textnormal{Int}B^m$ to $\B (u) = N u b_0$.
\end{proposition}

\begin{proof}
First we prove that $\pi:\Bbb{F}_0 \to \F$ is an injective map in each Bruhat cell in $\Bbb{F}_0$. Let $n$, $n' \in N$ and assume that $\pi (n u b_0) = \pi (n' u b_0)$. Then there is $c \in C$ such that $n u b_0 = n' u c b_0$, but since the Bruhat cells are disjoint then $c = 1$, so $n u b_0 = n' u b_0$ and $\pi$ is an injective map in each Bruhat cell.

Since $\pi$ is $G$-equivariant then $\pi (N u b_0) = N w b_\F$, where $\pi (u) = w \in W$, so that $\pi$ is also a surjective map in Bruhat cells, and then bijective in Bruhat cells.

Let us show that the map $\Psi_w|_{\textnormal{Int}B^m} : \textnormal{Int}B^m \to \B(w)$ is a surjective map, by induction on the length $\ell(w)$ of $w$. If $\ell(w)=1$, then the result follows from Lemma \ref{lemma:psiposto1}. For $\ell(w)>1$, let $y \in \B(w)$ and consider the canonical projection  $\pi_d:\F \to \F_d$, where $\F_d$ is the partial flag defined by the simple root associated to the simple reflection $r_d$. Since $\pi_d(y) \in \B(wr_d)$ and $\ell(wr_d)=d-1<\ell(w)$ we have, by the induction hypothesis,
\[
\pi_d(y)=\psi_1(t_1)c_1\cdots \psi_{d-1}(t_{d-1})c_{d-1}{b}_d,
\]
where $t_i \in \textnormal{Int}B^{m_i}$ and ${b}_d$ is the origin of $\F_d$. Note that
\[
wb_\F \in Nwb_\F \cap \pi_d^{-1}(w{b}_d).
\]
Applying Lemma \ref{lemma:psiposto1} to the fiber $\pi_d^{-1}(w{b}_d)$, which is the flag manifold of a rank one group, we have that $wb_\F=\psi_d(t_d)b_\F$. If $y'=\psi_1(t_1)c_1\cdots \psi_{d-1}(t_{d-1})c_{d-1}\psi_d(t_d)c_db_\F$, then
\[
\pi_d(y)=\psi_1(t_1)c_1\cdots \psi_{d-1}(t_{d-1})c_{d-1}{b}_d=\psi_1(t_1)c_1\cdots \psi_{d-1}(t_{d-1})c_{d-1}\psi_d(t_d)c_d{b}_d=\pi_d(y')
\]
and $y' \in Nwb_\F$ since $\psi_i(t_i) \not\in M(\alpha_i)$, from Proposition 1.5 of \cite{lonardo}. It follows by Lemma 1.4 of \cite{lonardo} that
\[
y=y'=\Psi_w(t_1,\ldots, t_d)
\]
since $Nwb_\F$ meets each fiber of $\pi_d$ in a unique point. Therefore, $\Psi_w$ is a surjective map.

Now, from Proposition 1.9 items (1),(2),(3) of \cite{lonardo} and using that $\pi \Psi_u = \Psi_w$ then
\begin{enumerate}
\item $\pi (\Psi_u (B^m)) = \pi (\S (u) )$.

\item $\pi (\Psi_u (\textbf{t}) )\in \pi ( \partial \B (u) )$ if and only if $\textbf{t} \in \partial B^m \cong \mathbb{S}^{m-1}$.

\item $\pi \Psi_u|_{\textnormal{Int}B^m} : \textnormal{Int}B^m \to N w b_\F = \pi (N u b_0) = \pi (\B (u))$ is a diffeomorphism.
\end{enumerate}

In order to show the corresponding statements for $\Bbb{F}_0$, we start with the third one. Since
\[
\pi(\Psi_u(\textnormal{Int}B^m)) = \pi(Nub_0)
\]
we have that
\[
 \Psi_u(\textnormal{Int}B^m) \subset \coprod_{c \in C} Nucb_0
\]
Take $N'= u' N (u')^{-1}$ where $u' = u^- u^{-1}$. In Proposition 2.7 of \cite{patrao-santos}
\[
N'=(N'\cap N^-)(N'\cap N)
\]
where in the paper the notation $N_u(B)$ is used for $N'$. Now note that
\[
u^-u^{-1}Nub_0= u^-u^{-1}Nu(u^-)^{-1}u^- b_0 = N'u^-b_0=(N'\cap N)(N'\cap N^-)u^-b_0
\]
and
\[
(N'\cap N)(N'\cap N^-)u^-b_0 =(N'\cap N)u^- (u^-)^{-1}(N'\cap N^-)u^-b_0 = (N' \cap N)u^-b_0 \subset Nu^-b_0
\]
since
\[
(u^-)^{-1}(N'\cap N^-)u^- \subset (u^-)^{-1}N^-u^- = N.
\]
It follows that
\[
 Nub_0 \subset u (u^-)^{-1} N u^-b_0 = uN^-b_0
\]
since $(u^-)^{-1}Nu^- = N^-$. Hence
\[
Nucb_0 \subset ucN^-b_0 = uN^-cb_0 = u (u^-)^{-1} N u^-cb_0 = u (u^-)^{-1} \B(u^-c).
\]
Since each $\B(u^-c)$ is open and they are disjoint and since $\Psi_u(\textnormal{Int}B^m)$ is connected and contains $ub_0$, then necessarily $\Psi_u(\textnormal{Int}B^m) = Nub_0$.
Hence
\[
\Psi_u(B^m) = \textnormal{cl} \Psi_u (\textnormal{Int}B^m) = \textnormal{cl} (\B (u)) = \S(u).
\]
Also $\Psi_u(\textbf{t}) \in \partial \S (u)$ if and only if $\textbf{t} \in \partial B^m  \cong \mathbb{S}^{m-1}$ and $\Psi_u|_{\textnormal{Int}B^m} : \textnormal{Int}B^m \to N u b_0 = \B (u)$ is a diffeomorphism.
\end{proof}

Let $u = s_1c_1 \cdots s_d c_d \in U$, where $c_i \in C$. Assume that $\pi(u)=r_1\cdots r_d$ and $w_i=r_1\cdots r_{i-1}r_{i+1}\cdots r_d$ are reduced expressions, for some $i\in \{1,\ldots, d\}$. Let us denote
\[
u_i^k = s_1 c_1\cdots s_{i-1}c_{i-1}s_i^{2k}c_is_{i+1}c_{i+1} \cdots s_d c_d
\]
where $k\in \{0,1\}$. If $m_i>1$, then $s^2_i=1$, so $u_i^1=u_i^0$.
We note that the maps $\Psi_{w_i}^k:\widehat{B}_i \to \F$, given by
\begin{eqnarray*}
    \Psi_{w_i}^k(t_1, \dots, t_{i-1}, t_{i+1}, \dots, t_d):= \psi_1 (t_1)c_1 \cdots \psi_i(t_{i-1})c_{i-1}s_i^{2k}c_i\psi_{i+1}(t_{i+1})c_{i+1}\cdots \psi_d (t_d)c_d b_\F
\end{eqnarray*}
where
\[
\widehat{B}_i = B^{m_1} \times \cdots \times B^{m_{i-1}}\times B^{m_{i+1}} \times \cdots \times B^{m_d},
\]
$k\in \{0,1\}$, satisfy the Proposition 1.9 of \cite{lonardo} and $\pi \circ \Psi_{u_i^k}=\Psi_{w_i}^k$.

We now define some subsets of $\partial B^{m}$. For this, let $i \in\{1,\ldots,d\}$. If $m_i>1$, then let
\begin{itemize}
    \item $B'_i = B^{m_1} \times \cdots \times B^{m_{i-1}}\times \{t'_i\} \times B^{m_{i+1}} \times \cdots \times B^{m_d}$,
\end{itemize} 
where $t'_i \in \mathbb{S}^{m_i-1}=\partial B^{m_i}$ is a fixed point.
If $m_i=1$, then we write $\wt{B}_i=B_i^0\cup B_i^1$, where
\begin{itemize}
    \item $B_i^0 = B^{m_1} \times \cdots \times B^{m_{i-1}}\times \{0\} \times B^{m_{i+1}} \times \cdots \times B^{m_d}$ ~and
    \item $B_i^1 = B^{m_1} \times \cdots \times B^{m_{i-1}}\times \{1\} \times B^{m_{i+1}} \times \cdots \times B^{m_d}$.
\end{itemize}
For each $i$, let
\begin{itemize}
    \item $\textnormal{Int}B'_i = \textnormal{Int}B^{m_1} \times \cdots \times \textnormal{Int}B^{m_{i-1}}\times \{t'_i\} \times \textnormal{Int}B^{m_{i+1}} \times \cdots \times \textnormal{Int}B^{m_d}$,
    \item $\textnormal{Int}B_i^0 = \textnormal{Int}B^{m_1} \times \cdots \times \textnormal{Int}B^{m_{i-1}}\times \{0\} \times \textnormal{Int}B^{m_{i+1}} \times \cdots \times \textnormal{Int}B^{m_d}$,
    \item $\textnormal{Int}B_i^1 = \textnormal{Int}B^{m_1} \times \cdots \times \textnormal{Int}B^{m_{i-1}}\times \{1\} \times \textnormal{Int}B^{m_{i+1}} \times \cdots \times \textnormal{Int}B^{m_d}$ ~and
    \item $\textnormal{Int}\wt{B}_i=\textnormal{Int}B^0_i \cup \textnormal{Int}B^1_i$.
\end{itemize}
In what follows, we will use the following notation $\partial B'_i = B'_i\backslash \textnormal{Int}B'_i$, $\partial B^0_i = B^0_i\backslash \textnormal{Int}B^0_i$, $\partial B^1_i = B^1_i\backslash \textnormal{Int}B^1_i$ and $\partial \wt{B}_i = \wt{B}_i\backslash \textnormal{Int}\wt{B}_i$.

\begin{lemma}\label{lemma-faces}
Let $u = s_1c_1 \cdots s_dc_d \in U$, for some $c_i \in C$, and denote
\begin{eqnarray*}
    u_i^k = s_1c_1 \cdots s_{i-1}c_{i-1}s_i^{2k}c_is_{i+1}c_{i+1} \cdots s_dc_d,
\end{eqnarray*}
$k\in \{0,1\}$. 
Suppose that $\pi(u) = w = r_1 \cdots r_d$ and $\pi(u_i^k) = w_i = r_1 \cdots r_{i-1}r_{i+1} \cdots r_d$ are reduced expressions.
\begin{enumerate}
   \item If $m_i=1$, then $\Psi_w$ is a diffeomorphism from $\textnormal{Int}B_i^0$ to $\Psi_w(\textnormal{Int}B_i^0)$ with
   \[
   \Psi_w(\textnormal{Int}B_i^0) =\Psi_w(\textnormal{Int}B_i^1) = \B(w_i)~   \mbox{ and }~
   \Psi_w(\partial B_i^0) =\Psi_w(\partial B_i^1) = \partial \B(w_i),
   \]
   and $\Psi_u$ is a diffeomorphism from $\textnormal{Int}\wt{B}_i$ to $\Psi_u(\textnormal{Int}\wt{B}_i)$ with
   \begin{eqnarray*}
        \Psi_u(\textnormal{Int}\wt{B}_i) = \B(u_i^0)\cup \B(u_i^1)
        \quad
        \mbox{and}
        \quad
        \Psi_u(\partial \wt{B}_i) = \partial \B(u_i^0)\cup \partial \B(u_i^1).
   \end{eqnarray*}
   \item If $m_i > 1$,
   then $\Psi_w$ is a diffeomorphism from $\textnormal{Int}B'_i$ to $\Psi_w(\textnormal{Int}B'_i)$ with
   \[
   \Psi_w(\textnormal{Int}B'_i)  = \B(w_i)
   \quad
   \mbox{and}
   \quad
   \Psi_w(\partial B'_i) =  \partial \B(w_i)
   \]
   and $\Psi_u$ is a diffeomorphism from $\textnormal{Int}B'_i$ to $\Psi_u(\textnormal{Int}B'_i)$ with
   \[
   \Psi_u(\textnormal{Int}B'_i) = \B(u_i^0)
   \quad
   \mbox{and}
   \quad
   \Psi_u(\partial B'_i) = \partial \B(u_i^0).
   \]
\end{enumerate}
\end{lemma}

\begin{proof}
If $m_i=1$, then it is immediate that
\[
\Psi_w (t_1, \dots, t_{i-1},0,t_{i+1}, \dots, t_d)
=
\Psi_{w_i}^0 (t_1, \dots, t_{i-1},t_{i+1}, \dots, t_d),
\]
\[
\Psi_w (t_1, \dots, t_{i-1},1,t_{i+1}, \dots, t_d)
=
\Psi_{w_i}^1 (t_1, \dots, t_{i-1},t_{i+1}, \dots, t_d),
\]
\[
\Psi_u (t_1, \dots, t_{i-1},0,t_{i+1}, \dots, t_d)
=\Psi_{u_i^0} (t_1, \dots, t_{i-1},t_{i+1}, \dots, t_d)
\]
and that
\[
\Psi_u (t_1, \dots, t_{i-1},1,t_{i+1}, \dots, t_d)
=\Psi_{u_i^1} (t_1, \dots, t_{i-1},t_{i+1}, \dots, t_d)
\]
which shows the first statement.

For the second assertion, we note that $\psi_i(t'_i)=1$, if $m_i>1$. Then it is immediate that
\[
\Psi_w (t_1, \dots, t_{i-1},t'_i,t_{i+1}, \dots, t_d)
=
\Psi_{w_i}^0 (t_1, \dots, t_{i-1},t_{i+1}, \dots, t_d)
\]
and that
\[
\Psi_u (t_1, \dots, t_{i-1},t'_i,t_{i+1}, \dots, t_d)
=
\Psi_{u_i^0} (t_1, \dots, t_{i-1},t_{i+1}, \dots, t_d),
\]
this concludes the proof.
\end{proof}

Let us consider the notation $B_i$ for $B^0_i$ or $B'_i$, depending on whether $m_i=1$ or $m_i>1$.

\begin{lemma} \label{lemma-order}
Let $X$ be a topological space and $\Psi: B^m \to X$ be a continuous map such that $\Psi$ is a homeomorphism from $\textnormal{Int}B^m$ to $\Psi(\textnormal{Int}B^m)$, that
\[
\Psi(\partial B^m) = \Psi(B^m)\backslash \Psi(\textnormal{Int}B^m).
\] 
If there is a subset $S$ of $\{1,\ldots,d\}$ such that, for each $i \in S$, $\Psi$ is a homeomorphism from $\textnormal{Int}B_i$ to $\Psi(\textnormal{Int}B_i)$, that
\[
\Psi(\partial B_i) = \Psi(B_i)\backslash \Psi(\textnormal{Int}B_i)
~~\text{ and ~ that }~~
\Psi(\partial B^m) = \bigcup_{i \in S} \Psi(B_i)
\]
then
\[
\bigcup_{j \notin S} \Psi(B_j) \subset \bigcup_{i \in S} \Psi(B_i)\backslash \Psi(\textnormal{Int}B_i).
\]
\end{lemma}

\begin{proof}
If $j \notin S$ and $(t_1,\ldots,t_d) \in B_j$, it follows that
\[
\Psi(t_1,\ldots,t_d) \in \Psi(\partial B^m) = \bigcup_{i \in S} \Psi(B_i).
\]
It remains to show that, if $i \in S$ and $(\tau_1,\ldots,\tau_d) \in \textnormal{Int}B_i$, then
\[
\Psi(t_1,\ldots,t_d) \neq \Psi(\tau_1,\ldots,\tau_d).
\]
Since $i \neq j$, we have that $\tau_j \in \textnormal{Int}B^{m_j}$ and thus there exists $\varepsilon > 0$ such that $B_{\varepsilon}(\tau_j) \subset \textnormal{Int}B^{m_j}$, where $B_{\varepsilon}(\tau_j)$ denotes the open ball with center $\tau_j$ and radius $\varepsilon$. If $i < j$, denote
\begin{eqnarray*}
    B &=& \textnormal{Int}B^{m_1} \times \cdots \times \textnormal{Int}B^{m_{i-1}}\times (\textnormal{Int}B^{m_{i}}\cup\{\tau_i\}) \times \textnormal{Int}B^{m_{i+1}} \times \cdots\\ 
    & & \cdots \times \textnormal{Int}B^{m_{j-1}}
    \times B_{\varepsilon}(\tau_j) \times \textnormal{Int}B^{m_{j+1}} \times \cdots  \times \textnormal{Int}B^{m_d}.
\end{eqnarray*}
If $j < i$, denote
\begin{eqnarray*}
    B &=& \textnormal{Int}B^{m_1} \times \cdots \times \textnormal{Int}B^{m_{j-1}}\times B_{\varepsilon}(\tau_j) \times \textnormal{Int}B^{m_{j+1}} \times \cdots \\
    & & \cdots  \times \textnormal{Int}B^{m_{i-1}} \times (\textnormal{Int}B^{m_{i}}\cup\{\tau_i\}) \times \textnormal{Int}B^{m_{i+1}} \times \cdots  \times \textnormal{Int}B^{m_d}.
\end{eqnarray*}
In both cases, we have that
\[
(\tau_1,\ldots,\tau_d) \in B \subset \textnormal{Int}B^m\cup\textnormal{Int}B_i
\]
and that $B$ is an open subset of $\textnormal{Int}B^m\cup\textnormal{Int}B_i$. Since $\Psi$ is a homeomorphism from $\textnormal{Int}B^m\cup\textnormal{Int}B_i$ to $\Psi(\textnormal{Int}B^m\cup\textnormal{Int}B_i)$, it follows that $\Psi(B)$ is an open neighborhood of $\Psi(\tau_1,\ldots,\tau_d)$ in $\Psi(\textnormal{Int}B^m\cup\textnormal{Int}B_i)$. 
For each $k\in\{1, \ldots, d\}$, let $t'_k \in \mathbb{S}^{m_k-1}=\partial B^{m_k}$, if $m_k>1$ and $t'_k=1$ if $m_k=1$. 
We have that
\[
\Psi(t_1,\ldots,t_d)
= \lim_{\delta \to 0^+} \Psi(t_1+\delta_1t'_1,\ldots, t_d + \delta_dt'_d)
\]
where $\delta_k  = -\delta$, if $t_k = 1$ or $t_k=t'_k$, and $\delta_k  = \delta$ otherwise. If we denote
\[
\delta_0 = d(\partial B_{\varepsilon}(\tau_j),\partial B^{m_j}) > 0,
\]
then for all $\delta \in (0,\delta_0)$, we have that
\[
(t_1+\delta_1t'_1,\ldots, t_d + \delta_dt'_d) \in \textnormal{Int}B^m \cup \textnormal{Int}B_i \backslash B
\]
since $t_j = 0$ or $t_j=t'_j$.
Hence $\Psi(t_1+\delta_1t'_1,\ldots, t_d + \delta_dt'_d) \notin \Psi(B)$ and therefore $\Psi(t_1,\ldots,t_d) \notin \Psi(B)$.
\end{proof}

\begin{proposition} \label{prop-order}
Let $w = r_1 \cdots r_d$ be a reduced expression and denote
\[
w_i = r_1 \cdots r_{i-1}r_{i+1} \cdots r_d.
\]
Let $S$ be the subset of $i \in \{1,\ldots,d\}$ such that $w_i$ is a reduced expression. Then
\[
\S(w) = \B(w) \cup \bigcup_{i \in S} \S(w_i)
\]
and
\[
\bigcup_{j \notin S} \S(w_j) \subset \bigcup_{i \in S} \partial \B(w_i).
\]
Furthermore, let $u = s_1c_1 \cdots s_dc_d $, for some $c_i \in C$, and denote
\[
u_i^k = s_1c_1 \cdots s_{i-1}c_{i-1}s_i^{2k}c_is_{i+1}c_{i+1} \cdots s_dc_d,
\]
$k\in\{0,1\}$. Then
\[
\S(u) = \B(u) \cup \bigcup_{i \in S} \S(u_i^0) \cup \S(u_i^1)
\]
and
\[
\bigcup_{j \notin S} \S(u_j^0) \cup \S(u_j^1) \subset \bigcup_{i \in S} \partial \B(u_i^0) \cup \partial \B(u_i^1).
\]
We note that if $m_l>1$, then $u_l^1=u_l^0$, so $\S(u_l^1)=\S(u_l^0)$ and $\partial \B(u_l^1)=\partial \B(u_l^0)$.
\end{proposition}

\begin{proof}
We have that
\[
\S(w)
=
\bigcup_{w'\leq w} \B(w')
=
\B(w) \cup \bigcup_{w' < w} \S(w').
\]
By the Proposition of Section 5.11 of \cite{humphreys}, we have that, if $w' < w$, then $w' \leq w_i$ for some $i \in S$. Since $\S(w') \subset \S(w_i)$, if $w' \leq w_i$, it follows that
\[
\S(w)
=
\B(w) \cup \bigcup_{i \in S} \S(w_i).
\]
We have that $\Psi_w: B^m \to \F$ is a continuous map such that $\Psi_w$ is a homeomorphism from $\textnormal{Int}B^m$ to $\Psi_w(\textnormal{Int}B^m)$, that
\[
\Psi_w(\textnormal{Int}B^m) = \B(w)
\qquad
\mbox{and}
\qquad
\Psi_w(B^m) = \S(w)
\]
and that
\[
\Psi_w(\partial B^m) = \partial \B(w) = \Psi_w(B^m)\backslash \Psi_w(\textnormal{Int}B^m).
\]
For each $i \in S$, $\Psi_w$ is a homeomorphism from $\textnormal{Int}B_i$ to $\Psi_w(\textnormal{Int}B_i)$, that
\[
\Psi_w(\textnormal{Int}B_i) = \B(w_i)
\qquad
\mbox{and}
\qquad
\Psi_w(B_i) = \S(w_i)
\]
and that
\[
\Psi_w(\partial B_i) = \partial \B(w_i) =\Psi_w(B_i)\backslash\Psi_w(\textnormal{Int}B_i).
\]
Since
\[
\Psi_w(\partial B^m) =
\partial \B(w) =
\bigcup_{i \in S} \S(w_i) =
\bigcup_{i \in S} \Psi_w(B_i)
\]
it follows that
\[
\bigcup_{j \notin S} \Psi_w(B_j) \subset
\bigcup_{i \in S}\Psi_w(B_i)\backslash\Psi_w(\textnormal{Int}B_i)
=
\bigcup_{i \in S} \partial \B(w_i).
\]
Since $w_j b_\F \in \Psi_w(B_j)$, for every $j \not\in S$, it follows that
\[
\bigcup_{j \notin S} \S(w_j) \subset \bigcup_{i \in S} \partial \B(w_i).
\]
From now on we will consider the notation $B_l$ for $\wt{B}_l$ or $B'_l$, depending on whether $m_l=1$ or $m_l>1$.
Denote by $p: \partial \B(u) \to \partial \B(w)$ the restriction to $\partial \B(u)$ of $\pi: \Bbb{F}_0 \to \F$.
Since
\[
\partial \B(u) = \Psi_u(\partial B^m) = \bigcup_{i = 1}^d \Psi_u(B_i),
\]
we have that
\[
\partial \B(u)
\backslash
\bigcup_{i \in S} \Psi_u(\textnormal{Int}B_i)
=
\bigcup_{j \notin S} \Psi_u(B_j)
\cup
\bigcup_{i \in S}\Psi_u(B_i)\backslash\Psi_u(\textnormal{Int}B_i). 
\]
Since
\[
p\left(\bigcup_{j \notin S} \Psi_u(B_j) \right)
=
\bigcup_{j \notin S} \Psi_w(B_j)
\subset
\bigcup_{i \in S} \partial \B(w_i)
\]
and since
\[
p\left(\bigcup_{i \in S} \Psi_u(B_i)\backslash \Psi_u(\textnormal{Int}B_i) \right)
=
\bigcup_{i \in S} \Psi_w(B_i)\backslash \Psi_w(\textnormal{Int}B_i)
=
\bigcup_{i \in S} \partial \B(w_i)
\]
it follows that
\[
p^{-1}\left(\bigcup_{i \in S} \B(w_i) \right)
=
\bigcup_{i \in S} \Psi_u(\textnormal{Int}B_i)
=
\bigcup_{i \in S} \B(u_i^0) \cup \B(u_i^1).
\]
Since $\bigcup_{i \in S} \B(w_i)$ is dense in $\partial \B(w)$ and since $\pi$ is a covering map, it follows that $\bigcup_{i \in S} \B(u_i^0) \cup \B(u_i^1)$ is dense in $\partial \B(u)$, which implies that
\[
\S(u) \backslash \B(u)
=
\partial \B(u)
=
\bigcup_{i \in S} \S(u_i^0) \cup \S(u_i^1).
\]
Finally, since
\[
p\left(\bigcup_{j \notin S} \S(u_j^0) \cup \S(u_j^1)\right)
=
\bigcup_{j \notin S} \S(w_j)
\subset
\bigcup_{i \in S} \partial \B(w_i)
\]
it follows that
\[
\bigcup_{j \notin S} \S(u_j^0) \cup \S(u_j^1) \subset \bigcup_{i \in S} \partial \B(u_i^0) \cup \partial \B(u_i^1),
\]
concluding the proof.
\end{proof}

Now the characterization of the order is almost immediate.

\begin{theorem} \label{theo-order}
Let $u, u' \in U$. Then $u' < u$ if and only if for some (or equivalently for each) reduced expression $\pi(u) = r_1 \cdots r_d$ and $u = s_1 \cdots s_dc$, for some $c \in C$, then $u'=s_1^{k_1}\cdots s_d^{k_d}c$, where
\[
k_i=
\left\{
\begin{array}{lcc}
     0 ~or~ 2, & if & i \in \{i_1, \ldots, i_l\} \\
     1, & if & i \notin \{i_1, \ldots, i_l\}
\end{array}
\right.
\]
with $w_0=\pi(u),$ $w_l=\pi(u')$ and $w_k$
is a reduced expression for each $0\leq k \leq l$, where
\[
w_k=\prod_{i \notin \{i_1,\ldots,i_k\}}r_i.
\]
\end{theorem}

Since we have characterized  the order of Morse components of the dynamics induced by regular elements, we can now characterize the order of Morse components of the dynamics induced by any arbitrary element of $G$. First we need a preliminary result.

\begin{lemma}
    If $H \in \textnormal{cl}\mathfrak{a}^+$ and $u\in U$, then
    \[
    G(H)ub_0=K(H)ub_0=K_H^0ub_0.
    \]
\end{lemma}

\begin{proof}
    By Proposition 3.3 of \cite{ps1} we have that $G_H^0ub_0 = K_H^0ub_0$.
    Since $K_H=K(H)M$, it follows that $K_H^0=K(H)M_0$ and that
    \[
    K_H^0ub_0 =K(H)M_0ub_0=K(H)ub_0,
    \]
    for all $u \in U$. Since $K(H) \subset G(H) \subset G_H^0$, we have that
    \[
    K(H)ub_0 \subset G(H)ub_0 \subset G_H^0ub_0 = K_H^0ub_0 = K(H)ub_0,
    \]
    which completes the proof.
\end{proof}

The Morse components of $H \in \textnormal{cl}\mathfrak{a}^+$ are given by $\mathcal{M}^H(u)=K_H^0ub_0$, where $u \in U$, and their Bruhat and Schubert cells are given by  $\B^H(u)=N_HK_H^0ub_0$ and $\S^H(u)=\cl \left(\B^H(u)\right)$.

\begin{proposition}
    If $H \in \textnormal{cl}\mathfrak{a}^+$ and $u \in U$, then 
    \[
    \B^H(u)=\bigcup_{v \in U_Hu}\B(v)
    \qquad
    \mbox{and}
    \qquad
    \S^H(u)=\bigcup_{v \in U_Hu}\S(v).
    \]
\end{proposition}

\begin{proof}
    Since $N=N_HN(H)$ and $N(H) \subset G(H)$ it follows that
    \[
    NK_H^0ub_0=N_HK_H^0ub_0=\B^H(u),
    \]
    thus $\bigcup_{v \in U_Hu}\B(v) \subset \B^H(u)$. Since $\bigcup_{v \in U}\B(v) = K/M_0$ we have
    \[
    \bigcup_{v \in U_Hu}\B(v) = \B^H(u).
    \]
    As the union on the left hand side is finite,
    \[
    \S^H(u)=\cl\left(\B^H(u)\right)= \bigcup_{v \in U_Hu}\cl\left(\B(v)\right) =\bigcup_{v \in U_Hu}\S(v),
    \]
    which concludes the proof.
\end{proof}

Now the next theorem, which characterizes the
inverse of the order of the minimal Morse components of any $H \in \cl \mathfrak{a}^+$ acting on $K/M_0$, is almost immediate.

\begin{theorem}
    If $H \in \cl \mathfrak{a}^+$ and $u,v \in U,$ then the following conditions are equivalent:
    \begin{enumerate}
        \item $u\leq_H v$.
        \item $\S^H(u) \subset \S^H(v)$.
        \item For every $u' \in U_Hu$, there exists $v'\in U_Hv$ such that $\S(u') \subset \S(v')$.
        \item For every $u' \in U_Hu$, there exists $v'\in U_Hv$ such that $u'\leq v'$.
    \end{enumerate}
\end{theorem}

Since the Morse components, Bruhat and Schubert cells of an arbitrary element of $G$ are equal to the Morse components, Bruhat and Schubert cells of its hyperbolic component, we get the following result.

\begin{corollary}
    Let $g$ be an arbitrary element of $G$ and $H \in \cl \mathfrak{a}^+$ be its hyperbolic component. The inverse order of the minimal Morse components of $g$ in $\Bbb{F}_0$ and in $K$ is given by $\leq_H$.
\end{corollary}

\section{Order of control sets}\label{sec:controlsets}

First, we observe that the usual projection $K \to K/M_0$ induces a bijective map from the control sets of $S$ on $K$ onto the control sets of $S$ on $K/M_0$. Furthermore, this map takes the control set on $K$ determined by $u$ to the control set on $K/M_0$ determined by the same element $u$. Hence, the order of the control sets on $K$ coincides with the order of the control sets on $K/M_0$. In order to determine this order, we start introducing an order on the canonical group $U$.

If $\alpha \in \Pi(B)$ and $E_{\alpha} \in \g_{\alpha}$ is such that
$
\langle E_{\alpha}, \theta E_{\alpha}\rangle(B)=-2\pi^2/\langle H_{\alpha},H_{\alpha}\rangle(B),
$
then 
\[
\langle gE_{\alpha}, \left(\Ad(g)\theta\Ad(g)^{-1}\right) gE_{\alpha}\rangle(gB)
    =-2\pi^2/\langle H_{g\alpha},H_{g\alpha}\rangle(gB).
\]
Since $g\mathfrak{g}_\alpha = \mathfrak{g}_{g\alpha}$, we can choose $E_{g\alpha}$ in $\mathfrak{g}_{g\alpha}$ such that $E_{g\alpha}=gE_{\alpha}$. Hence, $F_{g\alpha}=gF_{\alpha}$. Thus, for all $t,$
\[
\exp\left(tF_{g\alpha}\right)=\exp\left(gtF_{\alpha}\right)=g\exp\left(tF_\alpha\right)g^{-1}.
\]
In particular,
\[
s_{g\alpha}=gs_\alpha g^{-1} ~\text{ which implies }~ s_{g\alpha}^k=gs_\alpha^k g^{-1},
\]
for each nonnegative integer $k$. Therefore, if $u \in U$ and $u(B)=s_{\alpha_1}^{k_1}\cdots s_{\alpha_d}^{k_d}c$, with $c \in C(B)$ and $k_i\geq 0$, then $u(gB)=gu(B)g^{-1} \in U(gB)$ and
\[
u(gB)=(gs_{\alpha_1}^{k_1}g^{-1})\cdots (gs_{\alpha_d}^{k_d}g^{-1})(gcg^{-1})=s_{g\alpha_1}^{k_1}\cdots s_{g\alpha_d}^{k_d}gcg^{-1},
\]
with $gcg^{-1} \in C(gB)$.

The proof of the next proposition is analogous to that given for Lemma 2.2 and Proposition 2.3 of \cite{patrao-santos}.

\begin{proposition}\label{prop:orderU}
    Let $g \in G$ and $u, u' \in U$. Then
    \begin{enumerate}
        \item $u'(B) \leq u(B)$ if and only if $u'(gB)\leq u(gB)$.
        \item The relation $\leq$ on the canonical group $U$, defined by $u'\leq u$ if and only if $u'(gB) \leq u(gB)$, is a partial order.
    \end{enumerate}
\end{proposition}

As in Subsection \ref{sec:controlsetK}, let us fix a Weyl chamber $\la_0$ in $G$ such that $S\cap \la_0 \neq \emptyset$, let us denote by $\Bbb{D}$ the invariant control set of $S$ on $K/M_0$ such that $b(B)=b_0 \in \Bbb{D}_0$ and by $\B(S)$ the set of cosets determined by $S$ and $\Bbb{D}$, i.e.,
\[
\B(S)=\{gB : \la(gB) \cap S \neq \emptyset ~\text{ and }~ b(gB) \in \Bbb{D}_0\}.
\]
The \emph{canonical subgroup of $U$ characteristic of $S$} is defined by
\begin{equation*}\label{eqU(S)}
U(S) = \{u \in U : \Bbb{D}(u) = \Bbb{D} \}.
\end{equation*}

\begin{theorem}\label{teoU(S)subgrupo}
$U(S)$ is a subgroup of $U$.
\end{theorem}

\begin{proof}
From Theorem \ref{teo:controleU}, $1$ belongs to $U(S)$. Let  $u,u' \in U(S)$. If $gB \in \B(S)$, then, by Theorem \ref{teo:controleU}, $b_{u}(gB),
b_{u'}(gB) \in \Bbb{D}_0$. By Theorem \ref{teo:controleU} and Lemma \ref{lema:fixtipou},
there exists $g'B \in \mathcal{B}(S)$ such that
\[
b(g'B)=b_u(gB)=b(u(gB)gB).
\]
By Lemma \ref{lema:fixtipou}, there exists $n \in N(g'B)$ such that $u(gB)gB = n g'B$ and so
\begin{equation*}
b_{uu'}(gB)=b_{u'}(u(gB)gB)=b_{u'}(ng'B)=nb_{u'}(g'B).
\end{equation*}
Now let $h \in \la(g'B) \cap S$. By Proposition \ref{prop:fixtipou}, we have
$h^{-k}nh^{k} \rightarrow 1$, when $k \rightarrow \infty$.
Then there exists $l \in \N$ such that $(h^{-l}nh^{l})b_{u'}(g'B) \in
\Bbb{D}_0$, since $b_{u'}(g'B)\in \Bbb{D}_0$ and $\Bbb{D}_0$ is an open subset. By the $S$-invariance of $\Bbb{D}_0$, it follows that
\begin{equation*}
b_{uu'}(gB)=nb_{u'}(g'B)=h^l(h^{-l}nh^{l})b_{u'}(g'B)\in \Bbb{D}_0.
\end{equation*}
Therefore, $uu' \in U(S)$. Since $U$ is finite, this concludes the proof.
\end{proof}

\begin{proposition}\label{prop:projU(S)}
The following claims hold:
    \begin{enumerate}
        \item If $u \in U(S)$, then $\pi(u) \in W(S)$.
        \item If $w \in W(S)$, then there exist $u \in U$ and $c \in C$ such that $uc \in U(S)$ and $\pi(uc)=w$.
        \item If $C(S)=\{c \in C:\Bbb{D}(c)=\Bbb{D}\}$, then $C(S)$ is a normal subgroup of $U(S)$ and $U(S)/C(S)=W(S)$.
    \end{enumerate}
\end{proposition}

\begin{proof}
    If $u \in U$, then $\pi(\Bbb{D}(u)_0)=\Bbb{C}(\pi(u))_0$, by Theorem \ref{teo:controleU}. Thus, if $u \in U(S)$, then $\pi(u) \in W(S)$, by Theorem \ref{teo:controleU}, which shows 1.

    In order to show item 2, let $u \in U$ be such that $\pi(u)=w$. If $w \in W(S)$, then, by Theorem \ref{teo:controleU}, $\pi(\Bbb{D}(u)_0)=\Bbb{C}(w)_0=\Bbb{C}_0$ and so $\Bbb{D}(u)$ is an invariant control set, by Proposition \ref{prop:S-ics}. Thus, by Proposition \ref{prop:S-ics} and Theorem \ref{teo:controleU}, there exists $c \in C$ such that
    \[
    \Bbb{D}(uc)=\Bbb{D}(u)c=\Bbb{D}. 
    \]
    Therefore, $uc \in U(S)$ and $\pi(uc)=w$. 

    Finally, to show 3, let us note that the map $U(S)\to W(S)$ obtained by the restriction of $\pi: U \to W$ is a surjective group homomorphism and its kernel is $C \cap U(S)=C(S)$.
\end{proof}

The partial order $\leq$ on $U$ induces a relation on the set of right cosets $U(S)\backslash U$, saying that $U(S)u_1 \leq U(S)u_2$ if for every $u$ in $U(S)u_2$ there exists $u'$ in $U(S)u_1$ such that $u' \leq u$.

\begin{proposition}\label{prop:ordem}
    The relation $\leq$, defined above, is a partial order on $U(S)\backslash U$.
\end{proposition}

\begin{proof}
    It is easy to see that $\leq$ is reflexive and transitive. Suppose that
    \[
    U(S)u_1 \leq U(S)u_2 ~~\text{and}~~ U(S)u_2 \leq U(S)u_1.
    \]
    Let $u$ be a minimal element in $U(S)u_2$. Thus, there exists $\tilde{u} \in U(S)u_1$ such that $\tilde{u}\leq u$. Furthermore, there exists $u'\in U(S)u_2$ such that $u' \leq \tilde{u}$. Hence, $u'\leq \tilde{u}\leq u$. By the minimality of $u$ in $U(S)u_2$, we have $u'=\widetilde{u}=u$ and, thus, $U(S)u_1\cap U(S)u_2\neq \emptyset$, which implies that $U(S)u_1 = U(S)u_2$. Therefore, the relation $\leq$ is skew-symmetric, this concludes the proof.
\end{proof}

\begin{lemma}\label{lema:ordercontrolN}
    Let $u, u' \in U$ and $gB, g'B \in \B(S)$. If there exists $n \in N(gB)$ such that $b_{u'}(g'B)=nb_u(gB)$, then $\Bbb{D}(u) \leq \Bbb{D}(u')$.
\end{lemma}

\begin{proof}
    Let $h \in \la(gB) \cap S$. By Proposition \ref{prop:fixtipou}, we have $h^{-k}nh^{k} \rightarrow 1$, when $k \rightarrow \infty$. Hence, there exists $l \in \N$ such that $(h^{-l}nh^{l})b_{u}(gB) \in \Bbb{D}(u)_0$, since $b_{u}(gB)\in \Bbb{D}(u)_0$, by Theorem \ref{teo:controleU}, and $\Bbb{D}(u)_0$ is open. By Proposition \ref{prop:fixtipou}, $b_u(gB)$ is a fixed point of $h^l$ and so
    \[
    h^{-l}nb_{u}(gB)=(h^{-l}nh^{l})b_{u}(gB) = b_{u}(h^{-l}nh^{l}gB)\in \Bbb{D}(u)_0. 
    \] 
    Taking $g=h^l \in S$, we have 
    \[ 
    gb_{u}(h^{-l}nh^{l}gB)=nb_{u}(gB)=b_{u'}(g'B) \in \Bbb{D}(u')_0, 
    \] 
    by Theorem \ref{teo:controleU}, which shows that $\Bbb{D}(u)\leq \Bbb{D}(u')$.
\end{proof}

\begin{lemma}\label{lemma:contrclasseparab}
	If  $s \in U(S)$ and $u \in U$, then $\Bbb{D}(su) = \Bbb{D}(u)$. 
\end{lemma}

\begin{proof}
    Let $gB \in \mathcal{B}(S)$. By Theorem \ref{teo:controleU}, $b_{s}(gB) \in \Bbb{D}(s)_0=\Bbb{D}_0$. By Lemma \ref{lema:fixtipou} and Theorem \ref{teo:controleU}, this implies that there exists $g'B \in \mathcal{B}(S)$ such that
    \[
    b(g'B)=b_s(gB)=b(s(gB)gB).
    \]
    Hence there exists $n \in N(g'B)$ such that $s(gB)gB = n g'B$, so
    \begin{equation*}
    b_{su}(gB)=b_{u}(s(gB)gB)=b_{u}(ng'B)=nb_{u}(g'B),
    \end{equation*}
    from Lemma \ref{lema:fixtipou}. 
    Therefore, by Lemma \ref{lema:ordercontrolN}, $\Bbb{D}(u)\leq \Bbb{D}(su)$. Since $s^{-1} \in U(S)$ and $su \in U$ it follows that
    \[
    \Bbb{D}(su)\leq \Bbb{D}(s^{-1}(su))=\Bbb{D}(u),
    \]
    this concludes the proof.
\end{proof}

To establish a relationship between the order of control sets and the order on the set of cosets  $U(S)\backslash U$ we will use the following result.

\begin{proposition}\label{prop:fixtipow-fechoN}
    Let $u,u' \in U$ and $gB \in \B(S)$. Then $u \leq u'$ if and only if $b_{u}(gB) \in \mathrm{cl}(N(gB)b_{u'}(gB))$.
\end{proposition}

\begin{proof}
    Let us note that 
    $b_u(gB)=gu(B)b_0$, $b_{u'}(gB)=gu'(B)b_0$ and 
    \[
    \mathrm{cl}(N(gB)b_{u'}(gB))=\mathrm{cl}(gN(B)u'(B)b_0)=g\mathrm{cl}(N(B)u'(B)b_0).
    \]
    Furthermore, $u\leq u'$ if and only if
    \[
    \mathrm{cl}(N(B)u(B)b_0)=\S(u)\subset \S(u') = \mathrm{cl}(N(B)u'(B)b_0).
    \]
    Thus, if $u\leq u'$, then
    \[
    b_u(gB)=gu(B)b_0 \in g\mathrm{cl}(N(B)u'(B)b_0)=\mathrm{cl}(N(gB)b_{u'}(gB)).
    \]
    Conversely, if $b_u(gB) \in \mathrm{cl}(N(gB)b_{u'}(gB))$, then $u(B)b_0 \in \mathrm{cl}(N(B)u'(B)b_0)$. Thus there exists a sequence $(n_k)$ in $N(B)$ such that $n_ku'(B)b_0\to u(B)b_0$. Let $y=nu(B)b_0$ be an arbitrary point in $N(B)u(B)b_0$. Since $nn_ku'(B)b_0 \to nu(B)b_0=y$ it follows that $y \in \mathrm{cl}(N(B)u'(B)b_0)$ and thus $N(B)u(B)b_0 \subset \mathrm{cl}(N(B)u'(B)b_0)$. Then, 
    \[
    \S(u)=\mathrm{cl}(N(B)u(B)b_0)\subset \mathrm{cl}(N(B)u'(B)b_0)=\S(u').
    \]
    Therefore, $u\leq u'$.
\end{proof}

The following result establishes a relationship between the dynamical order of control sets of $S$ on $K/M_0$ with the order on the set of cosets $U(S)\backslash U$. This is a partial algebraic characterization, similar to what happens on the maximal flag manifold (see \cite{smo}).

\begin{theorem}
    If $U(S)u\leq U(S)u'$, then $\Bbb{D}(u') \leq \Bbb{D}(u)$. Furthermore, if $\Bbb{D}(u') \leq \Bbb{D}(u)$ and $\widehat{u}=s_1\cdots s_d \in U(S)u'$ is such that $\pi(\widehat{u})=r_1 \cdots r_d$ is a reduced expression, then $s_1^{k_1}\cdots s_d^{k_d} \in U(S)u$, for some $k_i\in \{0,1,2,3\}$.
\end{theorem}

\begin{proof}
    Suppose $U(S)u \leq U(S)u'$. Then there exists $u_1$ in $U(S)u$ such that $u_1 \leq u'$. By Lemma \ref{lemma:contrclasseparab}, $\Bbb{D}(u_1)=\Bbb{D}(u)$. Let $gB \in \B(S)$. By Proposition \ref{prop:fixtipow-fechoN}, 
    \[
    b_{u_1}(gB) \in \mathrm{cl}\left(N(gB)b_{u'}(gB)\right).
    \]
    Since $b_{u_1}(gB) \in \Bbb{D}(u_1)_0$ it follows that $N(gB)b_{u'}(gB)\cap \Bbb{D}(u_1)_0\neq \emptyset$, since $\Bbb{D}(u_1)_0$ is open. Hence, there exists $n \in N(gB)$ such that $b_{u'}(ngB)=nb_{u'}(gB) \in \Bbb{D}(u_1)_0$. By Theorem \ref{teo:controleU}, there exists $g'B \in \B(S)$ such that $nb_{u'}(gB)=b_{u_1}(g'B)$. Therefore, by Lemma \ref{lema:ordercontrolN}, $\Bbb{D}(u')\leq \Bbb{D}(u_1)=\Bbb{D}(u)$.

    Conversely, suppose that $\Bbb{D}(u') \leq \Bbb{D}(u)$. Let $\widehat{u} \in U(S)u'$. By Lemma \ref{lemma:contrclasseparab},
    \[ 
    \Bbb{D}(\widehat{u})=\Bbb{D}(u')\leq \Bbb{D}(u). 
    \] 
    Let $gB \in \B(S)$. By Theorem \ref{teo:controleU}, $b_{\widehat{u}}(gB) \in \Bbb{D}(\widehat{u})_0$ and $b_{u}(gB) \in \Bbb{D}(u)_0$. Hence, there exists $\widehat{g} \in S$ such that $b_{u}(gB) =\widehat{g} b_{\widehat{u}}(gB)$. Let $\widehat{u}=s_1\cdots s_d$ be such that $\pi(\widehat{u})=r_1\cdots r_d$ is a reduced expression, where $r_i=r_{\alpha_i}$ and $\alpha_i$ is a simple root. Let $\pi_d:K/M_0 \to K/K(\alpha_d)M_0$ be the canonical projection. We have that $\widehat{g}$ takes the fiber of $b_{\widehat{u}}(gB)$ into the fiber of $b_{u}(gB)$. Furthermore, we have that $b_{u}(gB)$, $b_{us_d}(gB)$, $b_{us_d^2}(gB)$ and $b_{us_d^3}(gB)$ are in the same fiber of $\pi_d$ and are the only fixed points, in this fiber, of $h \in \la(gB)\cap S$. Note that if $m_d>1$, then $s_d^2=1$ and $s_d^3=s_d$, so that we have in fact only two fixed points in this case. In all cases one of these fixed points is an attractor of the action of $h$ in this fiber of $\pi_d$. Since $\widehat{g}b_{\widehat{u}s_d^{-1}}(gB)$ is in this fiber and since the transitivity sets are open, it follows that there exists $l \in \N$ such that
    \[
    h^l\widehat{g}b_{\widehat{u}_{d-1}}(gB) \in \Bbb{D}(u)_0 \cup \Bbb{D}(us_d)_0 \cup \Bbb{D}(us_d^2)_0 \cup \Bbb{D}(us_d^3)_0,
    \]
    where $\widehat{u}_i=s_1\cdots s_i$ and thus $\widehat{u}_{d-1}=\widehat{u}s_d^{-1}$. Therefore, there exists $g_{d-1} \in S$ such that 
    \[ 
    g_{d-1}b_{\widehat{u}_{d-1}}(gB)=b_{us_d^{-k_d}}(gB) 
    \] 
    for some $k_d\in \{0,1,2,3\}$. Applying the same argument with $\pi_{d-1}$ in place of $\pi_d$ and $b_{\widehat{u}_{d-1}}(gB)$ in place of $b_{\widehat{u}}(gB)$, there exists $g_{d-2} \in S$ such that
    \[
    g_{d-2}b_{\widehat{u}_{d-2}}(gB)=b_{us_d^{-k_d}s_{d-1}^{-k_{d-1}}}(gB)
    \]
    for some $k_{d-1}\in \{0,1,2,3\}$. Continuing,
    \[
    b_{us_d^{-k_d}s_{d-1}^{-k_{d-1}}\cdots s_{1}^{-k_{1}}}(gB)=g_0b_{\widehat{u}_0}(gB)=g_0b(gB) \in \Bbb{D}_0.
    \]
    Therefore, $\widetilde{u}=us_d^{-k_d}s_{d-1}^{-k_{d-1}}\cdots s_{1}^{-k_{1}} \in U(S)$ so that
    \[
    s_1^{k_1}\cdots s_d^{k_d}=\widetilde{u}^{-1}u \in U(S)u,
    \]
    this concludes the proof.
\end{proof}

Finally, the next result establishes a bijection between the control sets of $S$ on $K/M_0$ with the set of cosets $U(S)\backslash U$.

\begin{corollary}\label{cor:cardconjcontK}
    The map given by $U(S)u \mapsto \Bbb{D}(u)$ is a well-defined and a bijective map between the sets of cosets $U(S)\backslash U$ and the control sets of $S$ on $K/M_0$.
\end{corollary}

\begin{proof}
    The map is well-defined by Lemma \ref{lemma:contrclasseparab} and it is surjective by Theorem \ref{teo:controleU}. Since $W=U/C$ and $W(S)=U(S)/C(S)$, by Proposition \ref{prop:projU(S)}, it follows that the number of elements of $U(S)\backslash U$ is the product of the cardinalities of $W(S)\backslash W$ and $C(S)\backslash C$, which is the number of control sets of $S$ on $K/M_0$, by Theorem 3.25 of \cite{patrao-santos}, which concludes the proof.
\end{proof}

\end{document}